\documentclass[11pt]{amsart}

\usepackage{amsmath}
\usepackage{tikz}
\usepackage{amsfonts}
\usepackage{amssymb}
\usepackage{amsthm}
\usepackage{mathdots}
\usepackage{graphicx}
\usepackage{pgf,tikz}
\usepackage{nicematrix}
\graphicspath{{images/}}
\usepackage{hyperref}
\usepackage[capitalize,nameinlink,noabbrev,nosort]{cleveref}
\usepackage{comment}
\crefname{figure}{Figure}{Figures}
\usepackage{float}
\usepackage{caption}
\usepackage{subcaption}

\numberwithin{equation}{section}

\newtheorem{thm}{Theorem}[section]
\newtheorem{lem}[thm]{Lemma}

\newtheorem{definition}[thm]{Definition}
\newtheorem{example}[thm]{Example}

\newtheorem{remark}[thm]{Remark}

\crefname{thm}{Theorem}{Theorems}
\crefname{prop}{Proposition}{Propositions}
\crefname{cor}{Corollary}{Corollaries}
\crefname{example}{Example}{Examples}
\crefname{definition}{Definition}{Definitions}
\crefname{lem}{Lemma}{Lemmas}
\crefname{section}{Section}{Sections}
\DeclareMathOperator{\Pf}{Pf}
\DeclareMathOperator{\sgn}{sgn}

\title
[High-dimensional cylindrical, toroidal, Möbius  and Klein grids]
{The monopole-dimer model on high-dimensional cylindrical, toroidal, Möbius  and Klein grids
}
\author{Anita Arora}
\address{Anita Arora, Department of Mathematics, Indian Institute of Science, Bangalore  560012, India.}
\email{anitaarora@iisc.ac.in/ anitaarora2018@iitkalumni.org}

\date{May 2024}

\begin{document}

\begin{abstract}
The dimer (monomer-dimer) model deals with weighted enumeration of perfect matchings (matchings).
The monopole-dimer model is a signed variant of the monomer-dimer model whose partition function is a determinant.
 In 1999, Lu and Wu~\cite{LUWu_1999} evaluated the partition function of the dimer model on two-dimensional grids embedded on a M\"{o}bius strip and a Klein bottle.
 While the partition function of the dimer model has been known for 
the two-dimensional grids with different boundary conditions,
we present a similar product formula for the partition function of the monopole-dimer model on higher dimensional cylindrical and toroidal grid graphs. We also evaluate the same for the three-dimensional M\"{o}bius and Klein grid graphs and show that the formula does not generalise for the higher dimensions. Further, we present a relation between the product formula for the three-dimensional cylindrical and M\"{o}bius grid.
\end{abstract}


\maketitle

\section{Introduction}
The \emph{dimer model} is the study of the physical process of adsorption of diatomic molecules (like oxygen) on a solid surface. Its partition function can be interpreted as enumerating weighted perfect matchings in an edge-weighted graph. Kasteleyn~\cite{Kasteleyn1963} solved the problem completely for planar graphs, by showing that the partition function of the dimer model can be written as a Pfaffian of a certain adjacency matrix built using a special class of orientations called Pfaffian orientations on the graph. One immediate consequence of Kasteleyn's result is that the Pfaffian is unaffected by the choice of orientation on the planar graph.
For the case of two-dimensional grid graphs $Q_{m,n}$, Kasteleyn~\cite{KASTELEYN19611209}
and Temperley--Fisher \cite{Fisher,TemperleyFisher} independently gave an explicit product formula for the partition function. 
 For example, when $m$ and $n$ are even,
horizontal (resp. vertical) edges have weight $a$ (resp. $b$), 
 the partition function of the dimer model on $Q_{m,n}$ with free boundary condition can be written as
\begin{equation}
\label{eqn:dimer-pf}
 \prod_{i=1}^{m/2} \prod_{j=1}^{n/2}
\left( 4a^2 \cos^2 \frac{i \pi}{m+1} + 4b^2 \cos^2 \frac{j \pi}{n+1}
\right).
\end{equation}
This formula is remarkable because although each factor is a degree-two polynomial in $a$ and $b$ with coefficients that may not be rational, the product turns out to be a polynomial with nonnegative integer coefficients. In particular, when $a = b = 1$, 
the result is an integer.

A similar product formula for the weighted enumeration of the perfect matchings has been given by McCoy and Wu~\cite{McCoy&Wu_1973} for the cylindrical and toroidal boundary conditions and by Lu and Wu~\cite{LUWu_1999} for the M\"{o}bius and Klein boundary conditions.
Tesler~\cite{TESLER2000} showed that the partition function of the dimer model on graphs embedded on non-orientable surfaces can be enumerated as a linear combination of some Pfaffians.
Brankov and Priezzhev~\cite{BranandPrie1993} gave explicit expressions for the free energy of the dimer model on finite quadratic lattices embedded on a Möbius strip. 

Efforts have been made to generalise and extend the dimer model while preserving this elegant structure. The natural physical generalisation is the \emph{monomer-dimer model}, which represents adsorption of a gas cloud consisting of both monoatomic and diatomic molecules. 
In the more abstract sense, it is the weighted enumeration of all matchings in a graph. 
Heilmann and Lieb~\cite{Heilmann_Lieb_1972} showed that the monomer-dimer model does not exhibit phase transitions.
However, this problem is known to be computationally difficult to handle~\cite{jerrum1987} and the partition function associated with it lacks a simple formula. A lower bound for the partition function of the monomer-dimer model for $d$-dimensional grid graphs has been obtained by Hammersley–Menon~\cite{hammersley-menon-1970} by generalising the method of Kasteleyn and Temperley–Fisher.

In another direction, a signed variant of the monomer-dimer model called the \emph{loop-vertex model} has been introduced by Ayyer~\cite{Ayyer2015ASM} for oriented graphs.
Configurations of the loop-vertex model can be thought of as superposition of two monomer-dimer configurations (matchings) having monomers (unmatched vertices) at the same locations. Consequently, loop-vertex configurations consist of even loops and isolated vertices. The loop-vertex model is less physical because of the presence of signs in their weights. On the other hand, the partition function here can be expressed as a determinant.
Ayyer also provided an orientation-independent interpretation of this model called the monopole-dimer model for planar graphs. This interpretation has been extended by the author and Ayyer~\cite{Arora_Ayyer2023} for the Cartesian product of planar graphs. 

The aim of this paper is to generalise the product formulas for the partition function of the dimer model on two-dimensional grids embedded on different surfaces to higher dimensions. We begin by introducing some notations and background results in \Cref{sec:BG}. We define high-dimensional cylindrical and toroidal grid graphs and compute the partition function of the monopole-dimer model on them in \Cref{sec:HDOrienGrid}. We also give the product formula for the partition function of the monopole-dimer model on the three-dimensional grids with M\"{o}bius and Klein boundary condition in \Cref{thm:PF3grid_Mob,thm:PF3grid_Klein}.  We show that the formulas do not hold for higher dimensions by providing counterexamples in \Cref{exm:CEMobHD,exm:CEKleinHD}. Further, we establish a relationship between three-dimensional grids with cylindrical and M\"{o}bius boundary conditions in \Cref{thm:Cyl_mob_relation}.
\section{Background}
\label{sec:BG}
A \textit{graph} is an ordered pair $G = (V(G), E(G))$, where $V(G)$ is the set of \textit{vertices} of $G$ and $E(G)$ is a collection of two-element subsets of $V(G)$, known as \textit{edges}. 
We will only consider finite undirected labelled graphs.
Unless stated otherwise, graphs will have the natural labelling $\{1,2,\dots,|V(G)|\}$. 
An \textit{orientation} on a graph $G$ is the assignment of arrows to its edges. A graph $G$ with an orientation $\mathcal{O}$ is called an \textit{oriented graph} and will be denoted as $(G,\mathcal{O})$. A \textit{canonical orientation} on a labelled graph is obtained by orienting its edges from a lower labelled vertex to a higher labelled vertex. 
Recall that a \textit{planar} graph is a graph which can be drawn in such a way that no edges will cross each other. Such an embedding of a planar graph is referred as a \textit{plane graph} and it divides the whole plane into regions, each of which is called a \textit{face}.

\begin{definition}
    Let $G_1$ and $G_2$ be two graphs. The \textit{Cartesian product of} $G_1$ and $G_2$ is the graph denoted $G_1 \square G_2$ with vertex set $V(G_1)\times V(G_2)$ and edge set  
\begin{equation*}
\left\{((u_1,u_2),(u'_1,u'_2)) \middle\vert
\begin{aligned}
\text{either } u_1 = u'_1 \text{ and } (u_2,u'_2) \in E(G_2)  \\ 
\text{or } u_2 = u'_2 \text{ and } (u_1,u'_1) \in E(G_1) 
\end{aligned}
\right\}.
\end{equation*}
\end{definition}
The above definition generalises to the \textit{Cartesian product of $k$ graphs} $G_1,\dots, G_k$, denoted $G_1 \square \allowbreak\cdots \square G_k$.
Throughout the text, we will denote the \emph{path} graph and the \emph{cycle} graph on $n$ vertices as $P_n$ and $C_n$, respectively. We write $Q_{n_1,\dots,n_d}$ for the $d$-dimensional grid graph which is the Cartesian product $P_{n_1}\square\cdots \square P_{n_d}$. 

\begin{definition}\cite[Definition~3.5]{Arora_Ayyer2023}
\label{def:ortd Cart}
The \emph{oriented Cartesian product} of naturally labeled oriented graphs $(G_1,\mathcal{O}_1), \allowbreak \dots, (G_k,\mathcal{O}_k)$ 
is the graph $G_1\square \cdots \square G_k$ with orientation $\mathcal{O}$ given 
as follows. For each $i \in [k]$, if $u_i\rightarrow u_i^\prime $ in $\mathcal{O}_i$, then 
$\mathcal{O}$ gives orientation
$(u_1,\dots,u_i,\dots,u_k) \rightarrow (u_1,\dots,u_i^\prime,\dots,u_k)$ if $u_{i+1}+u_{i+2}+\dots+u_k+(k-i)\equiv 0 \pmod 2$ and  $(u_1,\dots,u_i^\prime,\dots,u_k) \allowbreak \rightarrow (u_1,\dots,u_i,\dots,u_k)$ otherwise.
\end{definition}

The graph in \cref{fig:OrGon9} can be thought of as an oriented Cartesian product of path $P_3$ with itself which is labeled consecutively from one leaf to another.

\begin{definition}
An orientation on a plane graph $G$ is said to be \emph{Pfaffian} if it satisfies the property that each simple loop enclosing a bounded face has an odd number of clockwise oriented edges.
This property is also known as the \emph{clockwise-odd property}. 
\end{definition}
\begin{figure}[h]
    \centering
    \includegraphics[]{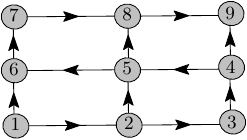}
    \caption{An oriented graph on $9$ vertices.}
    \label{fig:OrGon9}
\end{figure}

For example, the orientation in \cref{fig:OrGon9} is a Pfaffian orientation. Kasteleyn has shown that every plane graph posses a Pfaffian    orientation~\cite{Kasteleyn1963}. 
A \textit{dimer covering} (\emph{perfect matching}) is a collection of edges in the graph $G$ such that each vertex is covered in exactly one edge. 
The set of all dimer coverings of $G$ will be denoted as $\mathcal{M}(G)$.
Let $G$ be an edge-weighted graph on $2n$ vertices with edge-weight $w_e$ for $e \in E(G)$. Then the \textit{dimer model} is the collection of all dimer covers together with the weight of each dimer covering $M\in \mathcal{M}(G)$ given by $w(M)=\prod_{e \in M}w_e$.
The \textit{partition function} of the dimer model on $G$ is then defined as
\begin{equation*}
    \mathrm{Z}_G:=\sum_{M \in \mathcal{M}(G)} w(M),
\end{equation*}
which is basically the weighted enumeration of perfect matchings in $G$.
Before stating Kasteleyn's celebrated result, recall that the \textit{Pfaffian} of $2n \times 2n$ skew-symmetric matrix $A$ is given by
\begin{equation*}
\Pf(A)=\frac{1}{2^n n!} \sum_{\sigma \in  S_{2n}} \sgn(\sigma)\, A_{\sigma_1,\sigma_2} A_{\sigma_3,\sigma_4} \dots A_{\sigma_{2n-1},\sigma_{2n}}
\end{equation*}
and for such a matrix, $\det(A) = \Pf(A)^2$ by Cayley's theorem.

\begin{thm}[Kasteleyn~\cite{Kasteleyn1963}] 
If $G$ is a plane graph with Pfaffian orientation $\mathcal{O}$ and edge weight is $w_e$ for the edge $e=(u,v)$.
Then the partition function of the dimer model on $G$ is given by $\mathrm{Z}_G=\Pf(K_G)$,
where $K_G$ is a signed adjacency matrix defined by
\begin{align*}
    (K_G)_{u,v}&=\begin{cases}
        w_e & \text{if }u\rightarrow v \text{ in }\mathcal{O},\\
      -w_e &\text{if }v\rightarrow u \text{ in } \mathcal{O},\\
        \,0 &\text{otherwise.}
    \end{cases}
\end{align*}
\end{thm}
McCoy and Wu obtained a product formula for the two-dimensional grid graphs embedded on a cylinder and a torus similar to the one by Kasteleyn and Temperley-Fisher's formula in \eqref{eqn:dimer-pf} for the two-dimensional grid graphs with free boundary.

\begin{thm}[~\cite{McCoy&Wu_1973}]
\label{thm:DM_Cyl_Tor}
The partition function of the dimer model on the two-dimensional grid graph $Q_{2m,2n}$ where horizontal (resp. vertical) edges have weight $a$ (resp. $b$) with cylindrical and toroidal boundary conditions is given by
\begin{equation}
\label{eqn:2dCyl}
    \mathrm{Z}_{Q_{2m,2n}}^{\text{Cyl}}= \prod_{i=1}^{m} \prod_{j=1}^{n}
\left( 4a^2 \sin^2 \frac{(2i-1) \pi}{2m} + 4b^2 \cos^2 \frac{j \pi}{2n+1}
\right), 
\end{equation}
and 
\begin{equation}
\label{eqn:2dTor}
    \mathrm{Z}_{Q_{2m,2n}}^{\text{Tor}}= \prod_{i=1}^{m} \prod_{j=1}^{n}
\left( 4a^2 \sin^2 \frac{(2i-1) \pi}{2m} + 4b^2 \sin^2 \frac{(2j-1) \pi}{2n}
\right),
\end{equation}
respectively.
\end{thm}

\begin{figure}[h!]
\begin{subfigure}{.32\textwidth}
  \centering
\includegraphics[width=0.8\linewidth]{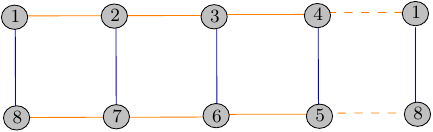}
  \caption{$4\times 2$ grid on a cylinder}
  \label{fig:2dCyl}
\end{subfigure}
\begin{subfigure}{.32\textwidth}
  \centering
  \includegraphics[width=0.8\linewidth]{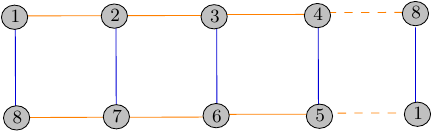}
  \caption{$4\times 2$ grid on a M\"{o}bius strip}
  \label{fig:2dMob}
\end{subfigure}
\begin{subfigure}{.32\textwidth}
  \centering
  \includegraphics[width=0.8\linewidth]{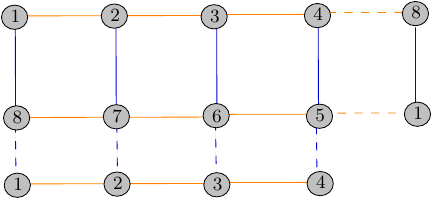}
  \caption{$4\times 2$ grid on a Klein}
  \label{fig:2dklein}
\end{subfigure}
\caption{2D grid with different boundary conditions}
\label{fig:2Dgrids_BC}
\end{figure}
Lu and Wu have obtained the similar closed-form expressions for the partition function of the dimer model on $2m\times 2n$ grids embedded on non-orientable surfaces like M\"{o}bius strip and Klein bottle.
\begin{thm}~\cite{LUWu_1999}
\label{thm:DM_mob_klein}
Let $Q_{2m,2n}=P_{2m}\square P_{2n}$ denote the two-dimensional grid graph with horizontal (resp. vertical) edges having weight $a$ (resp. $b$). The partition function of the dimer model on $Q_{2m,2n}$ embedded on a M\"{o}bius strip and on a Klein bottle is given by
\begin{equation}
\label{eqn:2dMob}    \mathrm{Z}_{Q_{2m,2n}}^{\text{M\"{o}b}}= \prod_{i=1}^{m} \prod_{j=1}^{n}
\left( 4a^2 \sin^2 \frac{(4i-1) \pi}{4m} + 4b^2 \cos^2 \frac{j \pi}{2n+1}
\right),
\end{equation}
and
\begin{equation}
\label{eqn:2dKlein}    \mathrm{Z}_{Q_{2m,2n}}^{\text{Klein}}= \prod_{i=1}^{m} \prod_{j=1}^{n}
\left( 4a^2 \sin^2 \frac{(4i-1) \pi}{4m} +  4b^2 \sin^2 \frac{(2j-1) \pi}{2n}
\right),
\end{equation}
respectively.   
\end{thm}
\Cref{fig:2Dgrids_BC} shows a two-dimensional grid embedded on different surfaces.

In this work, we will generalise 
\Cref{thm:DM_Cyl_Tor,thm:DM_mob_klein} for a more general model called the monopole-dimer model.
Let us first recall the loop-vertex model~\cite{Ayyer2015ASM}. Loops in the configurations will refer to directed cycles in the graph. We assume all the weights are real and positive. Let $G$ be a simple weighted graph on $n$ vertices with an orientation $\mathcal{O}$, vertex-weights $x(v)$ for $v\in V(G)$ and edge-weights $a_{v,v'}\equiv a_{v',v}$ for $(v,v') \in E(G)$. A
 \textit{loop-vertex configuration} $C$ of $G$
is a subgraph of the graph $G$ consisting of 
\begin{itemize}
    \item directed loops of even length (with length $\geq 4$),
    \item doubled edges (which can be thought of as loops of length 2),
    \item isolated vertices,
    \end{itemize}
with the condition that each vertex of $G$ is either covered in exactly one loop or is an isolated vertex. We will denote the set of all loop vertex configurations of $G$ as $\mathcal{L}(G)$. \cref{fig:Example of conf} shows a loop-vertex configuration on the grid graph $Q_{3,3}$ (see \Cref{fig:OrGon9}).
\begin{figure}[h]
\centering
\includegraphics[scale=1.1]{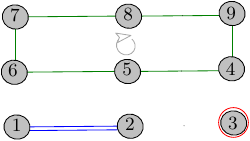}
\caption{A loop-vertex configuration on $Q_{3,3}$ consisting of a doubled edge at $(1,2)$, a directed cycle $(456789)$ and an isolated vertex at $3$. }
\label{fig:Example of conf}
\end{figure}

The \textit{sign} of an edge $(v,v') \in E(G)$, is defined as
\begin{equation}
\label{SE}
   \sgn(v,v'):=\begin{cases}
    1 & \text{ if }v \rightarrow v' \text{ in } \mathcal{O},
   \\ -1 & \text{ if }v' \rightarrow v \text{ in } \mathcal{O}.
    \end{cases}
\end{equation}
Let $\ell=(v_0,v_1,\dots,v_{2k-1},v_{2k}=v_0)$ be a directed even loop in $G$. The 
\textit{weight }of the loop $\ell$ is given by
\begin{equation}
\label{weightofloop}
w(\ell):=-\prod_{i=0}^{2k-1} \sgn(v_i,v_{i+1})\, a_{v_i,v_{i+1}}. 
\end{equation}
A loop-vertex configuration, $C=(\ell_1,\dots,\ell_j; \,v_1,\dots,v_k)$ consisting of loops $\ell_1,\dots,\ell_j$ and isolated vertices $v_1,\dots,v_k$, is given the weight
    \begin{equation}
    \label{eqn:weightofconf}
            w(C)=\prod_{i=1}^j w(\ell_i) \,\,\prod_{i=1}^k x(v_i).
    \end{equation}
Then the \textit{loop-vertex model} on the oriented graph $(G,\mathcal{O})$ is the collection $\mathcal{L}(G)$ where the weight of each configuration, $C\in \mathcal{L}(G)$ is assigned a weight as specified in \eqref{eqn:weightofconf}.
The \textit{(signed) partition function} of the loop-vertex model is defined as
\begin{equation*}
    \mathcal{Z}_{G,\mathcal{O}}:=\sum_{C \in \mathcal{L}(G)}w(C).
\end{equation*}

\begin{thm}\cite[Theorem~2.5]{Ayyer2015ASM}
\label{thm:pf-lvmodel}
The partition function of the loop-vertex model on $(G,\mathcal{O})$ is 
\begin{equation*}
    \mathcal{Z}_{G,\mathcal{O}}=\det{(\mathcal{K}_{G,\mathcal{O}})},
\end{equation*}
where $\mathcal{K}_{G,\mathcal{O}}$ is a generalised adjacency matrix of $(G,\mathcal{O})$ defined as:
\begin{equation}
\label{SAM}
    \mathcal{K}_{G,\mathcal{O}}(v,v')=\begin{cases}
    \,x(v) & \text{if }v=v',\\
    \, a_{v,v'} & \text{if } v\rightarrow v'\, in\, \mathcal{O},\\
    -a_{v,v'} & \text{if } v'\rightarrow v\, in\, \mathcal{O},\\
    \quad 0 & \text{if } (v,v') \notin E(G).
    \end{cases}
\end{equation}
We will use $\mathcal{K}_G$ instead of $ \mathcal{K}_{G,\mathcal{O}}$ whenever the underlying orientation is clear.
\end{thm}
\begin{example}
\label{PFLVM}
Let $G$ be a weighted graph on nine vertices with canonical orientation as shown in \cref{fig:OrGon9}. Let the vertex weights be $z$ for all the vertices and edge weights be $a_{i,j}\equiv a_{j,i}$ for the edge $(i,j)$. Then the weight of the configuration shown in \cref{fig:Example of conf} is $za_{1,2}^{2}a_{4,5}a_{5,6}a_{6,7}a_{7,8}a_{8,9}a_{4,9}$. The partition function of the loop-vertex model on the graph $G$ when all edge weights are equal to $a$, is 
$$
\mathcal{Z}_{G,\mathcal{O}} = z(8a^4 + 6a^2z^2 + z^4)^2
$$
which is also the determinant of the corresponding generalised adjacency matrix.
\end{example}
\begin{remark}
In the case of oriented Cartesian product of plane graphs each with a Pfaffian orientation, the loop-vertex model is known as the \emph{monopole-dimer model} and the 
weight of a loop $\ell=(v_0,v_1,\dots,v_{2k-1},\allowbreak 
v_{2k}=v_0)$ can be written \emph{independent of the orientation}~\cite[Theorem~3.8]{Arora_Ayyer2023}. In particular, for a plane graph, it can be expressed~\cite{Ayyer2015ASM} as
    \begin{equation}
    \label{MDMEQ}
        w(\ell)=(-1)^{\text{number of vertices enclosed by $\ell$}} \,\prod_{j=0}^{2k-1}  a_{v_j,v_{j+1}}.
    \end{equation}
\end{remark}   
Consequently, \cref{thm:pf-lvmodel} shows that the partition function of the monopole-dimer model on a Cartesian product of plane graphs is given by a determinant which turns out to be independent of the Pfaffian orientations on those plane graphs. The author and Ayyer~\cite[Theorem~6.1]{Arora_Ayyer2023} have extended the product formula in \eqref{eqn:dimer-pf} for the monopole-dimer model on the higher dimensional grids, we present a particular case of their result.

First, let us define the \emph{boustrophedon labelling}.
Recall that $P_n$ denotes the path graph on $n$  vertices and $Q_{n_1,\dots,n_d}$ is the $d$-dimensional grid graph with side lengths $n_1,\dots,n_d$ which can be regarded as the Cartesian product of $P_{n_1}, \dots, P_{n_d}$ denoted as $P_{n_1}\square\cdots\square P_{n_d}$. We will associate the boustrophedon labelling $L_d$ (defined inductively) on $Q_{n_1,\dots,n_d}$ as follows:

For $d=1$, label $L_1$ is $1,2,\dots,n_1$. For $d>1$, $Q_{n_1,\dots,n_d}$ consists of $n_d$ copies of $(d-1)$-dimensional grid graph $Q_{n_1,\dots,n_{d-1}}$. Successive copies (with successive last coordinate $1,2,\dots,n_d$) are labelled consecutively as $L_{d-1},L_{d-1}',L_{d-1},L_{d-1}',\dots$ where $L_{d-1}'$ represents the labelling of $(d-1)$-dimensional grid graph in reverse order of $L_{d-1}$. For example,
 the vertex $(p,q,r)$ in $Q_{2n_1,2n_2,2n_3}$ has label
\begin{equation}
\label{eqn:3DBousLabel}
\begin{cases}
8tn_1 n_2+4sn_1+p &  q=2s+1, r=2t+1,\\
      8tn_1 n_2+4sn_1-p+1 &  q=2s, r=2t+1,\\
      8tn_1 n_2-4sn_1-p+1 &  q=2s+1, r=2t,\\
      8tn_1 n_2-4sn_1+p &  q=2s, r=2t,
\end{cases}
\end{equation}
where $p \in [2n_1],q \in [2n_2]$ and $r \in [2n_3]$. 
\cref{fig:P4_P2_P2} shows this labelling on the graph $Q_{4,2,2}$. 
\begin{figure}[h!]
    \centering
    \includegraphics[scale=0.80]{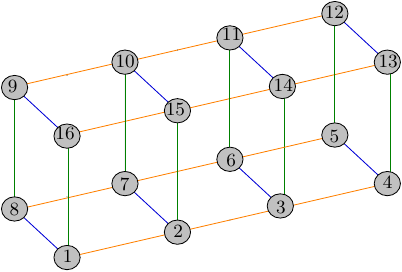}
    \caption{The boustrophedon labelling on $P_4\square P_2 \square P_2$.}
\label{fig:P4_P2_P2}
\end{figure}
Any snake-like labelling like the one above is a called a \emph{boustrophedon labelling}.
\begin{thm}\cite[Theorem~6.1]{Arora_Ayyer2023}
\label{thm:PFkDgrid}
Let $G$ be the $d$-dimensional grid graph $Q_{2m_1,\dots,2m_d}$ with boustrophedon labelling. {Let $(G,\mathcal{O})$ be obtained from $G$ 
by orienting the edges from a lower-labelled vertex to a higher-labelled vertex.} Let the vertex weights be $x$ for all vertices of $G$, and edge weights be $a_1,\dots,a_d$ for the edges 
along the different coordinate axes. 
Then the partition function of the monopole-dimer model on $G$ is given by
\begin{align}
\label{eqn:PFofkGrid}
    \mathcal{Z}_{G}\equiv\mathcal{Z}_{m_1,\dots,m_d}=
    \prod_{i_{1}=1}^{ m_1}\cdots \prod_{i_{d}=1}^{m_d} 
\left( x^2+\sum_{q=1}^{d} 4a_s^2\cos^2{\frac{i_{p_q}\pi }{2m_{p_q}+1}} \right)^{2^{d-1}}.
\end{align}
\end{thm}

In the following sections, we will extend the product formulas in \Cref{thm:DM_Cyl_Tor,thm:DM_mob_klein} for the monopole-dimer model on higher dimensional grids with different boundary conditions. 
This attempt parallels the approach 
in \Cref{thm:PFkDgrid} for higher dimensional grids with free boundary condition.

Let $T_k(-s,z,s)$ denote the $k\times k$ tridiagonal Toeplitz matrix with diagonal entries $z$, subdiagonal entries $-s$, and superdiagonal entries $s$, and let $J_k$ denote the $k\times k$ antidiagonal matrix with all antidiagonal entries equal to $1$. 
we will use the following notation for an $n\times n$ diagonal and antidiagonal matrix
\begin{align*}
    \mathrm{diag}\big(x_1,\dots ,x_n\big)&=\begin{pmatrix}
   x_1& &\\
     &\ddots &\\
    & & x_n
  \end{pmatrix},\\
    \mathrm{adiag}\big(x_1,\dots ,x_n\big)&=\begin{pmatrix}
   & &x_1\\
     &\iddots &\\
    x_n&& 
  \end{pmatrix}.
\end{align*}
Recall the definition of the \emph{Kronecker product} of two matrices.
\begin{definition}
\label{def:KronProdMat}
Let $A=(a_{i,j})$ be an $m\times n$ matrix and $B=(b_{i,j})$ be a $p\times q$ matrix, then the \emph{Kronecker product}, $A \otimes B$, is an $mp\times nq$ block matrix defined as
\[
\begin{pmatrix}
    a_{11}B &\cdots&a_{1n}B \\ 
    \vdots &\ddots&\vdots \\ 
    a_{m1}B &\cdots&a_{mn}B \\ 
\end{pmatrix}.
\]
\end{definition}
The Kronecker product of two matrices is non-commutative in general. If $A$ and $B$ are square matrices of order $n$ and $p$, respectively, then $$\det A\otimes B=(\det A )^p(\det B)^n.$$

Now, we recall some unitary similarity transforms that will be useful for computing determinants in the upcoming sections, along with their action on different matrices. We will use $\sim$ to denote the equivalence relation of similarity on matrices and $\iota$ for $\sqrt{-1}$.
\begin{lem}\cite[Section~4]{Fisher}
\label{lem:UTuk}
Let $u_k$ be the standard unitary similarity transformation whose entries are given by
\begin{equation}
    \label{eqn:UTuk}
    (u_k)_{p,q}=\sqrt{\frac{2}{k+1}}\iota^{p}\sin \left(
    \frac{pq\pi}{k+1} \right).
\end{equation}
Then, $u_k$
transforms the Toeplitz matrix, $T_k(-s,z,s)$, into the diagonal matrix 
$$
D_k =\mathrm{diag}\bigg(z+2\iota s\cos{\frac{\pi}{k+1}},
\dots,z+2\iota s\cos{\frac{k\pi}{k+1}}\bigg),
$$ 
and
\begin{equation}
\label{Eqn:UTforToeplitz}
(u_k)^{-1}J_k u_k=
\iota^{k-1}\begin{pmatrix}
   & & &(-1)^{k-1}\\
     & & (-1)^{k-2}&\\
     &\iddots & &\\
    (-1)^0& & & 
  \end{pmatrix}.
\end{equation}
\end{lem}
\begin{lem}\cite[Section~6]{McCoy&Wu_1973}
\label{lem:UTVn}
Let $V_n$ be the unitary similarity transform defined as 
\begin{equation}
\label{eqn:UTVn}
V_n=
\frac{1}{\sqrt{n}} \begin{pmatrix}
        1&\cdots&1\\
        e^{\iota\theta_1}&\cdots&e^{\iota\theta_{n}}\\
\vdots&&\vdots\\
        e^{(n-1)\iota\theta_1}&\cdots&e^{(n-1)\iota\theta_{n}}
    \end{pmatrix}
    \text{ for  }
    \theta_j=\frac{2j-1}{n}\pi.
\end{equation}
Then,
\begin{align*}
V_{n}^{-1}\left(T_{n}(-a,x,a)+a\mathrm{adiag}\left(1,0,\dots,0,-1\right)\right)V_{n}&=
\mathrm{diag}\left(x+2\iota a\sin{\frac{\pi}{n}},\dots, x+2\iota a\sin{\frac{(2n-1)\pi}{n}}\right),\\
V_{n}^{-1}J_{n}V_{n}&=
-\mathrm{adiag}\left(e^{\iota\frac{\pi}{n}},\dots,
e^{\iota\frac{(2n-1)\pi}{n}}
\right).
\end{align*}
\end{lem}
\section{High-dimensional 
 cylindrical and toroidal grid graphs}
\label{sec:HDOrienGrid}
First, recall that $P_n$ denotes the path graph on $n$  vertices and $Q_{n_1,n_2,n_3}$ is the three-dimensional grid graph. Throughout the text, we will assign $Q_{2n_1,2n_2,2n_3}$ the boustrophedon labelling 
defined in \eqref{eqn:3DBousLabel}.
\cref{fig:P4_P2_P2} shows this labelling on the graph $Q_{4,2,2}$. 
We will denote the following $n\times n$ antidiagonal matrices as
\begin{align*}
B_n^{Cyl}&=
\mathrm{adiag}\left(1,0,\dots,0,-1\right),
    \\
B_n^{\text{M\"{o}b}}&=\mathrm{adiag}\left(1,0,\dots,0,1\right).
\end{align*}

Let us now delve into the discussion regarding the partition function of the monopole-dimer model on higher dimensional grids with cylindrical and toroidal boundary conditions.
\begin{definition}
\label{defn:HDCylGrid}
We define an \emph{$\ell$-cylindrical grid} denoted $Q^{\ell}_{n_1,\dots,n_d}$ as the graph $C_{n_1}\square \cdots  \square  C_{n_{\ell}}\allowbreak\square P_{n_{\ell+1}}\square\cdots\square P_{n_d}$. 
For $\ell=1$ ($\ell=d$), we
 call it a \emph{cylindrical (toroidal) grid} and use the notation $Q^{\text{Cyl}}_{n_1,\dots,n_d}$ ($Q^{\text{Tor}}_{n_1,\dots,n_d}$).
\end{definition}
We sometimes refer to $Q^{\ell}_{n_1,\dots,n_d}$ as the $d$-dimensional grid $Q_{n_1,\dots,n_d}$ with cylindrical, toroidal and mixed boundary conditions depending on whether $\ell$ is $1,d$ or in between, respectively.
{Note that $Q^{\ell}_{n_1,\dots,n_d}$ with canonical orientation induced from boustrophedon labelling can be regarded as the oriented cartesian product of $C_{n_1}, \dots  ,\allowbreak C_{n_{\ell}},P_{n_{\ell+1}},\dots, P_{n_d}$. Thus, the loop-vertex model on an $\ell$-cylindrical grid is nothing but the monopole-dimer model.}
\begin{thm}
\label{thm:PFdgrid_cyl}
Let $G$ be the $d$-dimensional cylindrical grid $Q^{\text{Cyl}}_{2m_1,\dots,2m_d}$ with boustrophedon labelling. Let $(G,\mathcal{O})$ be obtained from $G$ 
by orienting the edges from lower-labelled vertex to higher-labelled vertex. Let the vertex weights be $x$ for all vertices of $G$, and edge weights be $a_1,\dots,a_d$ for the edges 
along the different coordinate axes. Then the partition function of the {monopole-dimer} model on $G$ is given by
        \begin{equation*}
     \mathcal{Z}_{2m_1,\dots,2m_d}^{\text{Cyl}}\equiv\mathcal{Z}_{G}^{\text{Cyl}}=\prod_{i_{1}=1}^{m_1}\cdots \prod_{i_{d}=1}^{m_d}\bigg(x^2+4a_1^2\sin^2{\frac{ (2i_1-1)\pi}{2m_1}}+\sum_{q=2}^{d} 4a_q^2\cos^2{\frac{i_{q}\pi }{2m_{q}+1}}\bigg )^{2^{d-1}}. 
   \end{equation*}
\end{thm}
\begin{figure}[h!]
    \centering    \includegraphics[scale=0.80]{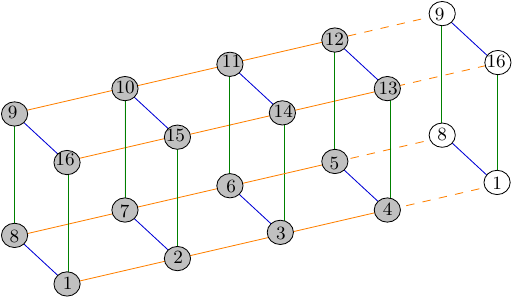}
    \caption{The boustrophedon labelling on the cylindrical grid $Q^{\text{Cyl}}_{4,2,2}$.}
    \label{fig:CylGrid}
\end{figure}
\Cref{fig:CylGrid} shows a three-dimensional grid graph with boustrophedon labelling and cylindrical boundary conditions. Using~\cite[Corollary 3.9]{Arora_Ayyer2023}, $\mathcal{Z}_{G}^{\text{Cyl}}$ remains independent of the various Pfaffian orientations on $C_{2n_1}, P_{2n_2},\dots,$ $P_{2n_d}$. Before presenting the proof, we exemplify the proof strategy below.
\begin{example}
Consider $G$ to be the three-dimensional cylindrical grid $Q^{\text{Cyl}}_{4,2,2}$ with the boustrophedon labelling as shown in \Cref{fig:CylGrid}. Orient $G$ in the canonical way.
Then the generalised adjacency matrix is
\begin{align*}
\mathcal{K}_G&=I_{2} \otimes I_{2}\otimes T_{4}(-a_1,x,a_1)
+I_{2}\otimes T_{2}(-a_2,0,a_2)\otimes J_{4} 
+T_{2}(-a_3,0,a_3)\otimes J_{2}\otimes J_{4}
\\
&
+a_1 I_{2}\otimes
    I_{2}\otimes B_4^{Cyl}.
\end{align*}
Let $u_2$ be the unitary matrix,
\[u_2=\sqrt{\frac{2}{3}}\begin{pmatrix}
    \iota\sin \frac{\pi}{3}
    &\iota\sin \frac{2\pi}{3}\\
    -\sin \frac{2\pi}{3}
    &-\sin \frac{4\pi}{3}
\end{pmatrix}.\]
Then using the unitary transform
$u_{2}\otimes u_{2} \otimes I_{4}$,  
we see that 
\begin{multline*}
  \mathcal{K}_G \sim 
  I_{2} \otimes I_{2} \otimes 
  \begin{pmatrix}
        x&a_1&0&a_1\\     -a_1&x&a_1&0\\
        0&-a_1&x&a_1\\
        -a_1&0&-a_1&x
    \end{pmatrix} 
  +I_{2} \otimes \begin{pmatrix}
   \iota a_2& 0\\
    0& -\iota a_2
  \end{pmatrix}\otimes J_{4}
+ \begin{pmatrix}
   \iota a_3& 0\\
    0& -\iota a_3
  \end{pmatrix}\otimes \iota\begin{pmatrix}
   0 &-1\\
   1& 0 
  \end{pmatrix}\otimes J_{4}.
\end{multline*}
Define, for $i_3 =1, 2$,
\begin{align*}
F_{i_3}&=I_{2} \otimes \begin{pmatrix}
        x&a_1&0&a_1\\     -a_1&x&a_1&0\\
        0&-a_1&x&a_1\\
        -a_1&0&-a_1&x
    \end{pmatrix} 
  + \begin{pmatrix}
   \iota a_2& (-1)^{i_3-1} a_3\\
    (-1)^{i_3}a_3& -\iota a_2
  \end{pmatrix}\otimes J_{4},
  \end{align*}
and $\det \mathcal{K}_G = \det F_{1} \det F_{2}$.
Now
    \begin{align*}
   F_{1} \sim  F_{2} \sim  &\, I_{2} \otimes \begin{pmatrix}
        x&a_1&0&a_1\\     -a_1&x&a_1&0\\
        0&-a_1&x&a_1\\
        -a_1&0&-a_1&x
    \end{pmatrix} 
  + \begin{pmatrix}
   \iota\sqrt{a_2^2+a_3^2}& 0\\
    0& -\iota\sqrt{a_2^2+a_3^2}
  \end{pmatrix}\otimes J_{4},
 \end{align*}
and thus both $F_1$ and $F_2$ have same the determinant. Hence $\det \mathcal{K}_G=\det F_{1}^2$.
Iterating the same procedure one more time, we get,
\begin{equation*}
\mathcal{Z}_{G}^{Cyl}
=\left(\det{F^{+}_{1,1}}\det{F^{-}_{1,1}}\right)^{2}, 
\end{equation*}
where 
\begin{equation*}
F^{\pm}_{1,1}= \begin{pmatrix}
        x&a_1&0&a_1\\     -a_1&x&a_1&0\\
        0&-a_1&x&a_1\\
        -a_1&0&-a_1&x
    \end{pmatrix} 
    \pm \iota \sqrt{a_2^2+a_3^2}
    J_{4}.
\end{equation*}
Let $V_4$ be the unitary similarity transform defined in \eqref{eqn:UTVn} for $n=4$,
then by \Cref{lem:UTVn}
\[
V_{4}^{*}F^{\pm}_{1,1}V_{4}=
\begin{pmatrix}
   x+2\iota a_1\sin{\frac{\pi}{4}}&0 & 0&\mp\iota e^{\frac{\pi}{4}\iota}\sqrt{a_2^2+a_3^2}\\
   \vspace{1mm}
    0&x+2\iota a_1\sin{\frac{3\pi}{4}}&\mp\iota e^{\frac{3\pi}{4}\iota}\sqrt{a_2^2+a_3^2}&0\\
    \vspace{1mm}
    0&\mp\iota e^{\frac{5\pi}{4}\iota}\sqrt{a_2^2+a_3^2}&x+2\iota a_1\sin{\frac{5\pi}{4}}&0\\
    \vspace{1mm}
   \mp\iota e^{\frac{7\pi}{4}\iota}\sqrt{a_2^2+a_3^2}&0&0& x+2\iota a_1\sin{\frac{7\pi}{4}}
  \end{pmatrix}.
\]
Thus, the partition function of the monopole-dimer model on $G$ is given by
\[
\mathcal{Z}_{G} =(x^2+2a_1^2+a_2^2+a_3^2)^8.
\]
\end{example}
\begin{proof}[Proof of \Cref{thm:PFdgrid_cyl}]
Define 
\[
M_j^d= \begin{cases}
I_{2m_d} \otimes\dots\otimes I_{2m_{2}}\otimes T_{2m_1}(-a_1,x,a_1) &  j = 1 \\
I_{2m_d} \otimes\dots\otimes I_{2m_{j+1}}\otimes T_{2m_j}(-a_j,0,a_j) \otimes J_{2m_{j-1}}\otimes \dots \otimes J_{2m_{1}} & 2\leq j\leq d.
\end{cases}
\]

Using \cref{thm:pf-lvmodel}, the partition function of the {monopole-dimer model} on $(G,\mathcal{O})$ is the determinant of the generalised adjacency matrix, $\mathcal{K}_G$, of $(G,\mathcal{O})$ and the generalised adjacency matrix, $\mathcal{K}_G$ with the boustrophedon labelling
can be written as

\begin{equation}
    \mathcal{K}_G =
    M_1^d+\dots +M_d^d
    +a_1 I_{2m_d} \otimes \cdots \otimes
    I_{2m_{2}}\otimes B_{2m_1}^{Cyl}.
\end{equation}
For $ j\in [d]\setminus\{1\}$, 
define the $2m_j\times 2m_j$ diagonal and antidiagonal matrices 
\begin{equation}
\label{eqn:DandJprime}
D_j = \mathrm{diag}\left(
   2\iota a_j\cos{\frac{\pi}{2m_j+1}},
     \dots, 2\iota a_j\cos{\frac{2m_j\pi}{2m_j+1}}\right) \quad \text{and }\quad
J_{j}^{\prime} = 
\iota^{2m_{j}-1}\mathrm{adiag}\left(
  -1,1,\dots,-1,1\right),
\end{equation}
and
  let 
\begin{align*}
  K_j^d&= 
I_{2m_d} \otimes\dots\otimes I_{2m_{j+1}}\otimes D_j \otimes J_{j-1}^{\prime}\otimes \dots \otimes J_{2}^{\prime}\otimes J_{2m_1},\\
\lambda_{i_d, i_{d-1},\dots,i_{p}}
&=\sqrt{\sum_{s=p}^{d}4a_{s}^2 \cos^2{\frac{i_{s}\pi}{2m_{s}+1}}}, \quad
  1< p\leq d.
\end{align*} 

Let 
$u_k$ be the similarity transformation whose entries are given by \eqref{eqn:UTuk}.
Then
using the unitary transform $u_{2m_d}\otimes\cdots \otimes u_{2m_2} \otimes I_{2m_1}$,  
we see that 
\begin{equation}
\label{eqn:simMatCyl}
    \mathcal{K}_G \sim I_{2m_d} \otimes \cdots \otimes I_{2m_2}\otimes \left(T_{2m_1}(-a_1,x,a_1)+a_1 B_{2m_1}^{Cyl}\right
    )
    +
    K_2^d +\cdots+K_{d}^d.
\end{equation}

Since the first matrix of each tensor product in \eqref{eqn:simMatCyl} is diagonal, $\mathcal{K}_G$ is similar to an $2m_d\times 2m_d$ block diagonal matrix with the block $F_{i_d}$, $i_d\in [2m_d]$, given by
\begin{multline*}
F_{i_d} = I_{2m_{d-1}} \otimes \cdots\otimes I_{2m_2}\otimes \left(T_{2n_1}(-a_1,x,a_1)+a_1 B_{2m_1}^{Cyl}\right
    )
    +
    K_2^{d-1} +\cdots+K_{d-2}^{d-1}\\
    +
\left(D_{d-1}
+2\iota a_d\cos{\frac{i_d\pi}{2m_d+1}} J_{{d-1}}^{\prime}\right) \otimes J_{d-2}^{\prime}\otimes\cdots\otimes J_{2m_{1}}.
  \end{multline*}
Diagonalizing the cruciform matrix (a matrix whose nonzero entries lie only on the diagonal and the antidiagonal)
\[
D_{d-1}+2\iota a_{d}\cos{\frac{i_d\pi}{2m_d+1}}J_{d-1}^{\prime}
\]
leads to the matrix
\[
\mathrm{diag}(\iota\lambda_{i_d,1},-\iota \lambda_{i_d,1},\iota\lambda_{i_d,2},-\iota \lambda_{i_d,2},\dots,\iota\lambda_{i_d,m_{d-1}},-\iota \lambda_{i_d,m_{d-1}}).
\]

Set $F_{i_d}^{+}=F_{i_d}$, $F_{i_d}^{-}=F_{2m_d-i_d+1}$ for $1\leq i_{d}\leq {m_{d}}$, and {observe that $F_{i_d}^{+}\sim F_{i_d}^{-}$.} 
Since the determinant of a matrix is invariant under similarity transformation,
the partition function can now be calculated as
\begin{equation}
\label{eqn:oneProdKD}
\mathcal{Z}_{G}^{Cyl}
=\prod_{i_d=1}^{m_{d}} (\det{F_{i_d}^{+}})^2.
 \end{equation}

Continuing with the above approach, $F_{i_d}^{+}$ is similar to an $2m_{d-1} \times 2m_{d-1}$ block diagonal matrix with blocks $F^{\pm}_{i_d,i_{d-1}}$ ($1\leq i_{d-1}\leq m_{d-1}$), 
resulting 
\begin{equation}
\label{eqn:twoProdKD}
\mathcal{Z}_{G}^{Cyl}
=\prod_{i_d=1}^{m_{d}}\prod_{i_{d-1}=1}^{m_{d-1}}  (\det{F_{i_d,i_{d-1}}^{+}})^4.
 \end{equation}
Inductively,
\begin{equation}
\label{eqn:mlProduct}
 \mathcal{Z}_{G}^{Cyl}
=\prod_{i_d=1}^{m_{d}}\cdots\prod_{i_{2}=1}^{m_{2}}(\det{F^{+}_{i_d,\dots,i_{2}}}\det{F^{-}_{i_d,\dots,i_{2}}})^{2^{d-2}},    
\end{equation}
where 
\begin{equation*}
F^{\pm}_{i_d,\dots,i_{2}}= \left(T_{2m_1}(-a_1,x,a_1)+a_1 B_{2m_1}^{Cyl}\right)
    \pm \iota\lambda_{i_d,\dots,i_{2}}
    J_{2m_1}.
\end{equation*}
Let $V_k$ be the unitary similarity transform defined in \eqref{eqn:UTVn}.
Then by \Cref{lem:UTVn},
\[
V_{2m_1}^{*}F^{\pm}_{i_d,\dots,i_{2}}V_{2m_1}=
\begin{pmatrix}
   x+2\iota a_1\sin{\frac{\pi}{2m_1}}& & &\mp\iota e^{\frac{\pi}{2m_1}\iota}\lambda_{i_d,\dots,i_{2}}\\
    &\ddots&\iddots&\\
   &\iddots&\ddots&\\
   \mp\iota e^{\frac{(4m_1-1)\pi}{2m_1}\iota}\lambda_{i_d,\dots,i_{2}}&&& x+2\iota a_1\sin{\frac{(4m_1-1)\pi}{2m_1}}
  \end{pmatrix},
\]
resulting,
\[
\det{F^{\pm}_{i_d,\dots,i_{2}}}=
\prod_{i_{1}=1}^{m_1}\bigg(x^2+4a_1^2\sin^2{\frac{ (2i_1-1)\pi}{2m_1}}+ \sum_{s=2}^{d}4a_{s}^2 \cos^2{\frac{i_{s}\pi}{2m_{s}+1}}\bigg).
\]
Substituting this in \eqref{eqn:mlProduct}, concludes the proof.
\end{proof}

\subsection{Grids with mixed boundary conditions}
Consider $d$-dimensional grids with mixed boundary conditions which are obtained from $d$-dimensional grids by imposing cylindrical boundary conditions in the first $\ell$ directions and keeping the rest free. A similar proof technique can be used to prove the following two theorems.
\begin{thm}
\label{thm:PFdgrid_Mix}
Let $G$ be the $\ell$-cylindrical grid graph $Q^{\ell}_{2m_1,\dots,2m_d}$ with boustrophedon labelling in $d$ dimension. {Let $(G,\mathcal{O})$ be obtained from $G$ 
by orienting the edges from a lower-labelled vertex to a higher-labelled vertex.} Let the vertex weights be $x$ for all vertices of $G$, and edge weights be $a_1,\dots,a_d$ for the edges 
along the different coordinate axes. Then the partition function of the {monopole-dimer} model on {$G$} is given by
        \begin{equation*}
     \mathcal{Z}_{2m_1,\dots,2m_d}^{\text{Mix}}=\prod_{i_{1}=1}^{m_1}\cdots\prod_{i_{d}=1}^{m_d}\left(x^2
     +\sum_{s=1}^{\ell}4a_s^2\sin^2{\frac{ (2i_s-1)\pi}{2m_s}}
     +\sum_{t=\ell+1}^d4a_t^2\cos^2{\frac{ i_t\pi}{2m_t+1}}\right)^{2^{d-1}}. 
   \end{equation*}
\end{thm}
\subsection{Grids with toroidal boundary conditions}

Let us now consider $d$-dimensional grids with toroidal boundary conditions which are obtained by imposing cylindrical boundary conditions in all directions of a $d$-dimensional grid.
\begin{thm}
\label{thm:PFdgrid_Tor}
Let $G$ be the toroidal grid graph $Q^{\text{Tor}}_{2m_1,\dots,2m_d}$ with boustrophedon labelling in $d$ dimension. {Let $(G,\mathcal{O})$ be obtained from $G$ 
by orienting the edges from a lower-labelled vertex to a higher-labelled vertex.} Let the vertex weights be $x$ for all vertices of $G$, and edge weights be $a_1,\dots,a_d$ for the edges 
along the different coordinate axes. Then the partition function of the {monopole-dimer}  model on {$G$} is given by
        \begin{equation*}
     \mathcal{Z}_{2m_1,\dots,2m_d}^{\text{Tor}}=\prod_{i_{1}=1}^{m_1}\cdots\prod_{i_{d}=1}^{m_d}\left(x^2+\sum_{s=1}^d4a_s^2\sin^2{\frac{ (2i_s-1)\pi}{2m_s}}\right)^{2^{d-1}}. 
   \end{equation*}
\end{thm}
For the rest of the paper, we will focus on higher dimensional Möbius and Klein grid graphs.
\section{High-dimensional 
M\"{o}bius grid graphs}
\label{sec:HDNOrienGrid}
In this section, we will extend the product formula \eqref{eqn:2dMob} for three-dimensional grids and show that the formula doesn't extend to higher dimensions in the obvious way. {We will use the notation $\Bar{A}$ and $A^{*}$ to denote the conjugate and conjugate transpose of the matrix $A$, respectively.}
\begin{lem}
\label{lem:det2by2Block}
    Let $A$ and $B$ be two square matrices of order $n$ such that
    $\overline{B}A^*=\overline{A} \,\overline{B}$, and $B^t=B$.
    Then, the determinant of the block matrix
\[
\det\begin{pmatrix}
    A& B\\
    -\overline{B}& \overline{A}
\end{pmatrix}=
\det{(AA^*+B{B}^*)}.
\]
\end{lem}
\begin{proof}
It is well known that
\begin{align*}
\det\begin{pmatrix}
    A& B\\
    C& D
\end{pmatrix}&=\det{(A-BD^{-1}C)}\det{D}\\
&=\det{(A-BD^{-1}C)}\det{D^t},\\
&=\det{(AD^t-BD^{-1}CD^t)}.
\end{align*}
Here, we have $D=\overline{A}$ and $C=-\overline{B}$, therefore
\begin{align*}
\det\begin{pmatrix}
    A& B\\
    -\overline{B}& \overline{A}
\end{pmatrix}
&=\det{(A\overline{A}^t+B\overline{A}^{-1}\overline{B}A^{*})}\\
&=\det{(A\overline{A}^t+B\overline{A}^{-1}\overline{A}B^{*})}\\
&=\det{(A{A}^{*}+BB^{*})}.
\end{align*}
\end{proof}
\begin{definition}
\label{def:HDMob}
Let $Q_{n_1, \dots,n_d}$ be the $d$-dimensional grid graph, we add an edge between the vertices $(1,k_2,\dots,k_d)$ and $(n_1,n_2-k_2+1,\dots,n_d-k_d+1)$ for all $1\leq k_i\leq n_i \,(2\leq i\leq d)$ to obtain $Q_{n_1, \dots,n_d}$ with \emph{M\"{o}bius boundary condition} along the first direction. We call these edges as \emph{dashed edges} and the remaining as \emph{solid edges}. We call this new graph as $d$-dimensional \emph{Möbius grid graph} and denote it as $Q^{\text{Möb}}_{n_1,\dots,n_d}$.
\end{definition}

Let $G=Q^{\text{Möb}}_{n_1,\dots,n_d}$ be the $d$-dimensional Möbius grid graph with boustrophedon labelling. 
Orient the solid edges from lower-labelled vertex to higher-labelled vertex, orient the dashed edge at $1$ outward and the remaining dashed edges such that each two-dimensional square satisfies the clockwise-odd property. Let us denote the resulting oriented graph as $(G,{\mathcal{O}})$.
We will always orient the edges coming from the M\"{o}bius boundary condition as described above.
\Cref{fig:MobGrid} shows such an orientation over the Möbius grid graph $Q^{\text{Möb}}_{4, 2,2}$.
\begin{definition}
We define the \emph{monopole-dimer model} on the $d$-dimensional Möbius grid graph $G$ as the loop-vertex model on $G$  with the above orientation $\mathcal{O}$. \emph{The partition function} of the monopole-dimer model is then the partition function of the loop-vertex model.
\end{definition}
\begin{thm}
\label{thm:PF3grid_Mob}
Let $G$ be the three-dimensional Möbius grid graph $Q^{\text{Möb}}_{2m_1, 2m_2,2m_3}$ with boustrophedon labelling. Let the vertex weights be $x$ for all vertices of $G$, and edge weights be $a_1,a_2$ and $a_3$ for the edges 
along the $x$-,$y$- and $z$- coordinate axes respectively. Then the partition function of the monopole-dimer model on $G$ is given by
        \begin{equation*}
     \mathcal{Z}_{2m_1,2m_2,2m_3}^{\text{Möb}}=\prod_{i_{1}=1}^{m_1}\prod_{i_{2}=1}^{m_2}\prod_{i_{3}=1}^{m_3}\bigg(x^2+4a_1^2\sin^2{\frac{ (4i_1-1)\pi}{4m_1}}+4a_2^2\cos^2{\frac{i_{2}\pi }{2m_{2}+1}}+4a_3^2\cos^2{\frac{i_{3}\pi }{2m_{3}+1}}\bigg )^{4}. 
   \end{equation*}
\end{thm}
\begin{figure}[h!]
    \centering
    \includegraphics[scale=0.80]{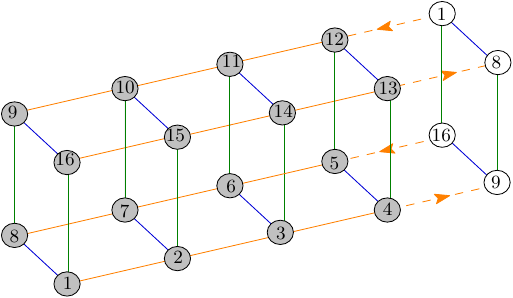}
    \caption{The three-dimensional M\"{o}bius grid graph $Q^{\text{Möb}}_{4, 2,2}$.}
    \label{fig:MobGrid}
\end{figure}
\begin{remark}
The product formula in \Cref{thm:PF3grid_Mob} remains unchanged even if one starts by orienting the dashed edge at $1$ inward and the remaining dashed edges such that each two-dimensional square satisfies the clockwise-odd property.
\end{remark}
\begin{remark}
    Note that the oriented $d$-dimensional Möbius grid $Q^{\text{Möb}}_{2m_1,\dots,2m_d}$ can be regarded as the oriented Cartesian product of $P_{2m_1},\dots,P_{2m_d}$ (oriented from one leaf to another) together with some additional dashed edges oriented in the specified way. We believe (though we have not been able to prove) that the partition function of the monopole-dimer model remains unchanged regardless of the orientation on the path graphs. That's the reason for naming it as the monopole-dimer model.
\end{remark}
\begin{proof}
    Let $D_j$ and $J_{j}^{\prime}$ be as defined in \eqref{eqn:DandJprime}. The generalised adjacency matrix $\mathcal{K}_G$ of the oriented three-dimensional M\"{o}bius grid graph $(G,\mathcal{O})$
is
\begin{multline*}
    \mathcal{K}_G = 
    I_{2m_3} \otimes I_{2m_2}\otimes  T_{2m_1}(-a_1,x,a_1)
    +
I_{2m_3} \otimes T_{2m_2}(-a_2,0,a_2) \otimes J_{2m_{1}}
+
T_{2m_3}(-a_3,0,a_3) \otimes J_{2m_2} \otimes J_{2m_{1}}
\\
+a_1 \mathrm{adiag}\left(
  1,-1,\dots,1,-1\right)
  \otimes \mathrm{diag}\left(
   1,-1
     \dots, 1,-1\right)\otimes B_{2m_1}^{\text{Möb}}
\end{multline*}
Let $u_k$ be the standard unitary similarity transformation defined in \Cref{lem:UTuk}.
Then using the unitary transform $u_{2m_3}\otimes u_{2m_2} \otimes I_{2m_1}$,  
it is clear that 
\begin{multline}
\label{eqn:KGMOB3}
    \mathcal{K}_G\sim 
    I_{2m_3} \otimes I_{2m_2}\otimes  T_{2m_1}(-a_1,x,a_1)
    +
I_{2m_3} \otimes D_{2} \otimes J_{2m_{1}}
+
D_{3} \otimes J_{2}^{\prime} \otimes J_{2m_{1}}
\\
+a_1 \iota^{2m_3-1} \mathrm{diag}\left(
   1,-1
     \dots, 1,-1\right)
    \otimes J_{2m_{2}}\otimes B_{2m_1}^{\text{Möb}}.
\end{multline}
Given that the initial matrix of each tensor product is diagonal in \eqref{eqn:KGMOB3}, $\mathcal{K}_G$ becomes similar to a $2m_3\times 2m_3$ block diagonal matrix, with each block $F_{i_3}$, ($i_3\in [2m_3]$), defined as
\begin{multline}
\label{eqn:Fi3Mob}
F_{i_3}=I_{2m_2}\otimes  T_{2m_1}(-a_1,x,a_1)
    +
 D_{2} \otimes J_{2m_{1}}
+
2\iota a_3\cos{\frac{i_3\pi}{2m_3+1}}J_{2}^{\prime} \otimes J_{2m_{1}}
\\
+a_1(-1)^{i_3-1} \iota^{2m_3-1}J_{2m_2}\otimes B_{2m_1}^{\text{Möb}}.
\end{multline}
Therefore,
\begin{equation}
\label{eqn:KisProdFi3}
\det{\mathcal{K}_G}=\prod_{i_3=1}^{2m_3}\det{F_{i_3}}.
\end{equation}

Applying some simultaneous row and column interchanges on the first matrix in each tensor of \eqref{eqn:Fi3Mob} results in
\begin{align*}
\label{eqn:FinMob}
F_{i_3}\sim & I_{m_2}\otimes I_{2}\otimes  T_{2m_1}(-a_1,x,a_1)
   \\
& +
    \mathrm{diag}\big(
    2\iota a_2\cos{\frac{\pi}{2m_2+1}},     \dots ,2\iota a_2\cos{\frac{m_2\pi}{2m_2+1}}\big)
  \otimes \begin{pmatrix}
   1& 0\\
   0 &-1
  \end{pmatrix} \otimes J_{2m_{1}}\\
&+
2\iota^{2m_2} a_3\cos{\frac{i_3\pi}{2m_3+1}} 
\mathrm{diag}\big(1,-1,1,-1,\dots)
\otimes\begin{pmatrix}
   0& -1\\
   1&0
  \end{pmatrix} \otimes J_{2m_{1}}
\\
&+a_1(-1)^{i_3-1} \iota^{2m_3-1}I_{m_2}\otimes J_{2}
\otimes B_{2m_1}^{\text{Möb}}.
\end{align*}
Consequently,
\begin{align*}
\det{F_{i_3}}=\prod_{i_2=1}^{m_2}
\det&\Big(I_{2}\otimes  T_{2m_1}(-a_1,x,a_1)
+    2\iota a_2\cos{\frac{i_2\pi}{2m_2+1}}
    \begin{pmatrix}
   1& 0\\
   0 &-1
  \end{pmatrix} \otimes J_{2m_{1}}
  \\
&+
2(-1)^{i_2-1}\iota^{2m_2} a_3\cos{\frac{i_3\pi}{2m_3+1}} 
\begin{pmatrix}
   0& -1\\
   1&0
  \end{pmatrix} \otimes J_{2m_{1}}
  +a_1(-1)^{i_3-1} \iota^{2m_3-1} J_{2}
  \otimes B_{2m_1}^{\text{Möb}}\Big).
\end{align*}
Using the unitary transform $\frac{1}{\sqrt{2}}\begin{pmatrix}
    -1&1\\
    1 &1
\end{pmatrix}\otimes I_{2m_1}$,
the determinant of $F_{i_3}$ can be expressed as the following product:
\begin{align*}
\det{F_{i_3}}=\prod_{i_2=1}^{m_2}
\det\Big(&I_{2}\otimes  T_{2m_1}(-a_1,x,a_1)
+    2\iota a_2\cos{\frac{i_2\pi}{2m_2+1}}
    \begin{pmatrix}
   0& -1\\
   -1 &0
  \end{pmatrix} \otimes J_{2m_{1}}
  \\
&+
2(-1)^{i_2-1}\iota^{2m_2} a_3\cos{\frac{i_3\pi}{2m_3+1}} 
\begin{pmatrix}
   0& 1\\
   -1&0
  \end{pmatrix} \otimes J_{2m_{1}}\\
  &+a_1(-1)^{i_3-1} \iota^{2m_3-1} \begin{pmatrix}
   -1& 0\\
   0&1
\end{pmatrix}\otimes B_{2m_1}^{\text{Möb}} \Big).
\end{align*}
Finally, notice that $F_{i_3}$ is the product of determinants of $2\times 2$ block matrices of the following form
\[
\begin{pmatrix}
    A& B\\
    -\overline{B}& \overline{A}
\end{pmatrix}
\]
where
\begin{equation*}
    A=T_{2m_1}(-a_1,x,a_1)-a_1(-1)^{i_3-1} \iota^{2m_3-1}  B_{2m_1}^{\text{Möb}}
\end{equation*}
and
\begin{equation*}
B= \bigg(-2\iota a_2\cos{\frac{i_2\pi}{2m_2+1}}+2(-1)^{i_2-1}\iota^{2m_2} a_3\cos{\frac{i_3\pi}{2m_3+1}} \bigg)
    J_{2m_{1}}.
\end{equation*}
Observe that
\begin{align*}
    \overline{B}A^*=&
    \bigg(2\iota a_2\cos{\frac{i_2\pi}{2m_2+1}}+2(-1)^{i_2-1}\iota^{2m_2} a_3\cos{\frac{i_3\pi}{2m_3+1}} \bigg)
    \\
    &\times J_{2m_{1}}\left(T_{2m_1}(a_1,x,-a_1)+a_1(-1)^{i_3-1} \iota^{2m_3-1}  B_{2m_1}^{\text{Möb}}\right)
\end{align*}
Since, $J_{2m_{1}}$ commutes with $B_{2m_1}^{\text{Möb}}$ and $J_{2m_{1}}T_{2m_1}(a_1,x,-a_1)=T_{2m_1}(-a_1,x,a_1)J_{2m_{1}}$ 
    \begin{align*}    \overline{B}A^*=&\bigg(2\iota a_2\cos{\frac{i_2\pi}{2m_2+1}}+2(-1)^{i_2-1}\iota^{2m_2} a_3\cos{\frac{i_3\pi}{2m_3+1}} \bigg)\\
    &\times
    \Bigg(T_{2m_1}(-a_1,x,a_1)J_{2m_{1}}+a_1(-1)^{i_3-1} \iota^{2m_3-1}  B_{2m_1}^{\text{Möb}}J_{2m_{1}}\Bigg)
\end{align*}
After re-arranging the terms, we get
\begin{align*}
    \overline{B}A^*=&\Bigg(T_{2m_1}(-a_1,x,a_1)+a_1(-1)^{i_3-1} \iota^{2m_3-1}  B_{2m_1}^{\text{Möb}}\Bigg)\\
    &\times\bigg(2\iota a_2\cos{\frac{i_2\pi}{2m_2+1}}+2(-1)^{i_2-1}\iota^{2m_2} a_3\cos{\frac{i_3\pi}{2m_3+1}} \bigg)J_{2m_{1}}
    \\
    =&\overline{A} \,\overline{B}.
\end{align*}
Now, using \Cref{lem:det2by2Block},
we get
\begin{equation}
\label{eqn:detABBconjAconj}
\det\begin{pmatrix}
    A& B\\
    -\overline{B}& \overline{A}
\end{pmatrix}=\det{(AA^*+B{B}^*)}.
\end{equation}
Diagonalizing the matrix $A$ yields the following:
$$A\sim \mathrm{diag}\bigg(x+2(-1)^{m_3+i_3-1}ia_1\sin{\frac{3\pi}{4m_1}},\dots,x+2(-1)^{m_3+i_3-1}\iota a_1\sin{\frac{(4(2m_1)-1)\pi}{4m_1}}\bigg),$$
and
\begin{equation}
\label{eqn:AAstar}
AA^*\sim \mathrm{diag}\bigg(x^2+4a_1^2\sin^2{\frac{3\pi}{4m_1}},\dots,x^2+4a_1^2\sin^2{\frac{(8m_1-1)\pi}{4m_1}}\bigg).
\end{equation}
Further, a quick calculation shows that
\begin{equation}
\label{eqn:BBstar}
BB^*=\bigg(4a_2^2\cos^2{\frac{i_2\pi}{2m_2+1}}+4a_3^2\cos^2{\frac{i_3\pi}{2m_3+1}} \bigg)I_{2m_1}.
\end{equation}
Combining equations \eqref{eqn:detABBconjAconj}, \eqref{eqn:AAstar}, and \eqref{eqn:BBstar} results
\begin{equation*}
\det{F_{i_3}}=\prod_{i_2=1}^{m_2}\prod_{i_1=1}^{2m_1} \bigg(x^2+4a_1^2\sin^2{\frac{(4i_1-1)\pi}{4m_1}}+4a_2^2\cos^2{\frac{i_2\pi}{2m_2+1}}+4a_3^2\cos^2{\frac{i_3\pi}{2m_3+1}}\bigg).
\end{equation*}
Finally, \eqref{eqn:KisProdFi3}
together with the following trigonometric identities
\begin{equation}
\label{eqn:sin4theta}
    \sin{\frac{(4(m_1+i_1)-1)\pi}{4m_1}}=-\sin{\frac{(4i_1-1)\pi}{4m_1}},
\end{equation}
\begin{equation}
\label{eqn:cos2theta}
    \cos{\frac{(2m_3-i_3+1)\pi}{2m_3+1}}=-\cos{\frac{i_3\pi}{2m_3+1}}
\end{equation}
concludes the proof.
\end{proof}

We now provide an example showing that the formula does not generalise for higher dimensions.
\begin{example}
\label{exm:CEMobHD}
    Let $G=Q^{\text{Möb}}_{2,2,2,2}$ be the four-dimensional M\"{o}bius grid as shown in \Cref{fig:4DMob}. The solid edges are oriented from lower labelled vertex to higher labelled vertex and dashed edges are oriented as described in the paragraph just below \Cref{def:HDMob}.
    \begin{figure}[h!]
        \centering        \includegraphics[scale=0.85]{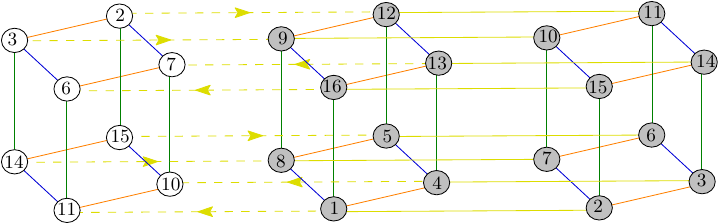}
        \caption{The four-dimensional cube with M\"{o}bius boundary conditions}
        \label{fig:4DMob}
    \end{figure}
    Let the vertex weight be $0$ for all the vertices and edge weights be $a_1,a_2,a_3$ and $a_4$ for the edges 
along the different coordinate axes. Then, the partition function of the monopole-dimer model on $G$ is 
\begin{equation*}
    (4a_1^2 + a_2^2 + a_3^2 + a_4^2)^4(a_2^2 + a_3^2 + a_4^2)^4
\end{equation*} 
which is not an $8^{\text{th}}$ power. This leads us to conclude that the product formula in \eqref{eqn:2dMob} does not generalise to higher dimensions.
\end{example}
Using
\Cref{thm:PFdgrid_cyl,thm:PF3grid_Mob} together with the following identities,
\[
\sin^2{\frac{(2(2i)-1)\pi}{4n}}=
\sin^2{\frac{(2(2n-2i+1)-1)\pi}{4n}}=\sin^2{\frac{(4i-1)\pi}{4n}}, \, \text{ for }1\leq i\leq n,
\]
we can deduce the generalized relationship between the partition function of the monopole-dimer model on the three-dimensional M\"{o}bius and cylindrical grid graphs, akin to the relationship between the partition function of the dimer model on two-dimensional grids embedded on a cylinder and a M\"{o}bius strip~\cite[(24)]{LUWu_1999}.
\begin{thm}
\label{thm:Cyl_mob_relation}
Let $\mathcal{Z}_{4n_1,2n_2,2n_3}^{\text{Cyl}}$ and $\mathcal{Z}_{2n_1,2n_2,2n_3}^{\text{Möb}}$ be the partition function of the monopole-dimer model on the three-dimensional M\"{o}bius grid $Q^{\text{Möb}}_{4n_1,2n_2,2n_3}$ and cylindrical grid $Q^{\text{Cyl}}_{2n_1,2n_2,2n_3}$ with boustrophedon labelling, respectively. Then
\begin{equation}
\label{eqn:Cyl_mob3D_relation}
\mathcal{Z}_{4n_1,2n_2,2n_3}^{\text{Cyl}}=
\left(\mathcal{Z}_{2n_1,2n_2,2n_3}^{\text{Möb}}\right)^2. 
\end{equation}
\end{thm}
\section{High-dimensional 
Klein grid graphs}
\label{sec:Klein_BC}
A $d$-dimensional grid is considered to have \emph{Klein boundary conditions} if it exhibits M\"{o}bius boundary conditions along the first direction and cylindrical boundary conditions along the remaining directions.
\begin{definition}
Let $P_{n_1}\square C_{n_2}\square\cdots\square C_{n_d}$ be the Cartesian product of $P_{n_1},C_{n_1},\dots, C_{n_d}$, we add an edge between the vertices $(1,k_2,\dots,k_d)$ and $(n_1,n_2-k_2+1,\dots,n_d-k_d+1)$ for all $1\leq k_i\leq n_i \,(2\leq i\leq d)$ to obtain a new graph called the $d$-dimensional \emph{Klein grid graph} $Q^{\text{Klein}}_{n_1, \dots,n_d}$. We call these additional edges as \emph{dashed edges} and the remaining as \emph{solid edges}.    
\end{definition}
Let $G=Q^{\text{Klein}}_{n_1,\dots,n_d}$ be the $d$-dimensional Klein grid graph with boustrophedon labelling. 
Orient the solid edges from lower-labelled vertex to higher-labelled vertex, orient the dashed edge at $1$ outward and the remaining dashed edges such that each two-dimensional square satisfies the clockwise-odd property. Let us denote the resulting oriented graph as $(G,{\mathcal{O}})$.
Note that $Q^{\text{Möb}}_{n_1, \dots,n_d}$ is a subgraph of $Q^{\text{Klein}}_{n_1, \dots,n_d}$.
\Cref{fig:KleinGrid} shows such an orientation over the Klein grid graph $Q^{\text{Klein}}_{4, 2,2}$.
\begin{figure}[h!]
    \centering
    \includegraphics[scale=0.80]{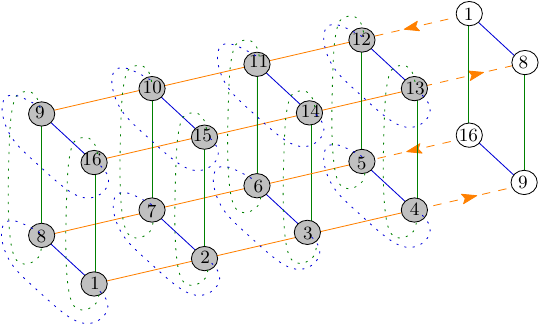}
    \caption{The three-dimensional M\"{o}bius grid graph $Q^{\text{Klein}}_{4, 2,2}$.}
    \label{fig:KleinGrid}
\end{figure}
\begin{definition}
We define the \emph{monopole-dimer model} on the $d$-dimensional Klein grid graph $G$ as the loop-vertex model on $G$  with the above orientation $\mathcal{O}$. \emph{The partition function} of the monopole-dimer model is then the partition function of the loop-vertex model.
\end{definition}
\begin{thm}
\label{thm:PF3grid_Klein}
Let $G=Q^{\text{Klein}}_{2m_1,2m_2,2m_3}$ be the three-dimensional Klein grid graph. Let vertex weights be $x$ for all vertices of $G$, and edge weights be $a_1,a_2$ and $a_3$ for the edges 
along the $x$-,$y$- and $z$- coordinate axes respectively. Then the partition function of the monopole-dimer model on $(G,\mathcal{O})$ is given by
\begin{equation*}
\mathcal{Z}_{2m_1,2m_2,2m_3}^{\text{Klein}}=\prod_{i_{1}=1}^{m_1}\prod_{i_{2}=1}^{m_2}\prod_{i_{3}=1}^{m_3}\bigg(x^2+4a_1^2\sin^2{\frac{ (4i_1-1)\pi}{4m_1}}+4a_2^2\sin^2{\frac{ (2i_2-1)\pi}{2m_2}}+4a_3^2\sin^2{\frac{ (2i_3-1)\pi}{2m_3}}\bigg )^{4}. 
   \end{equation*}
\end{thm}
\begin{proof}
    Again, let $T_k,J_k,D_j$ and $J_{j}^{\prime}$ be as defined in the proof of \Cref{thm:PFdgrid_cyl}.
    The generalised adjacency matrix, $\mathcal{K}_G$ of the oriented three-dimensional grid graph with Klein boundary conditions, $(G,\mathcal{O})$
is
\begin{multline*}
    \mathcal{K}_G = 
    I_{2m_3} \otimes I_{2m_2}\otimes  T_{2m_1}(-a_1,x,a_1)
+a_1 \mathrm{adiag}\left(
   1,-1
     \dots, 1,-1\right)
     \otimes \mathrm{diag}\left(
   1,-1
     \dots, 1,-1\right)
\otimes B_{2m_1}^{\text{Möb}}
    \\
        +
I_{2m_3} \otimes \left(T_{2m_2}(-a_2,0,a_2)
+a_2  B_{2m_2}^{Cyl}\right)\otimes J_{2m_1}
+\left(T_{2m_3}(-a_3,0,a_3) 
+a_3 B_{2m_3}^{Cyl}\right)\otimes J_{2m_2}\otimes J_{2m_1}.
\end{multline*}
Let $V_k$ be the similarity transformation defined in \eqref{eqn:UTVn}.
Then using the unitary transform $V_{2m_3}\otimes V_{2m_2} \otimes I_{2m_1}$,  
it is clear that 
\begin{align*}
\mathcal{K}_G\sim &
    I_{2m_3} \otimes I_{2m_2}\otimes  T_{2m_1}(-a_1,x,a_1)\\
&+a_1 \begin{pNiceArray}{ccc|ccc}[]
  &&e^{\iota\frac{\pi}{2m_3}}  & &  & \\
  &\iddots& &  &\text{\Huge{0}} &  \\
e^{\iota\frac{(2m_3-1)\pi}{2m_3}}&&&  & & \\
\hline
 &  &  &&&e^{\iota\frac{(2m_3+1)\pi}{2m_3}}\\
 &  \text{\Huge{0}}&  & &\iddots& \\
 &  &  & e^{\iota\frac{(4m_3-1)\pi}{2m_3}}& &\\
\end{pNiceArray}
\otimes 
\begin{pNiceArray}{c|c}[]
  0&  I_{m_2} \\
\hline
I_{m_2} & 0
\end{pNiceArray}
\otimes B_{2m_1}^{\text{Möb}}
\end{align*}
\begin{align*}
    +&
I_{2m_3} \otimes \mathrm{diag}\left(
    2\iota a_2\sin{\frac{(2\times1-1)\pi}{2m_2}},     \dots ,2\iota a_2\sin{\frac{(2\times2m_2-1)\pi}{2m_2}}\right)\otimes J_{2m_1}\\
+&\mathrm{diag}\left(
    2\iota a_3\sin{\frac{(2\times1-1)\pi}{2m_3}},     \dots ,2\iota a_3\sin{\frac{(2\times2m_3-1)\pi}{2m_3}}\right)\otimes 
    -\mathrm{adiag}\left(
    e^{\iota\frac{\pi}{2m_2}},
    \dots,
    e^{\iota\frac{(2\times2m_2-1)\pi}{2m_2}}
    \right)
    \\ 
    &\otimes J_{2m_1}.
\end{align*}
Using the following identities
\begin{align}
\label{eqn:TriId}
e^{\iota\frac{(2(n+k)-1)\pi}{2n}}&=-e^{\iota\frac{(2k-1)\pi}{2n}},\\
\sin{\frac{(2(n+k)-1)\pi}{2n}}
&=-\sin{\frac{(2k-1)\pi}{2n}},\nonumber
\end{align}
we get
\begin{multline*}
    \mathcal{K}_G\sim 
    I_2\otimes I_{m_3} \otimes I_{2m_2}\otimes  T_{2m_1}(-a_1,x,a_1)\\
+a_1 \begin{pmatrix}
    1&0\\
    0 &-1
\end{pmatrix}\otimes
\mathrm{adiag}\left(
e^{\iota\frac{\pi}{2m_3}},
\dots,
e^{\iota\frac{(2m_3-1)\pi}{2m_3}}
\right)
\otimes 
\begin{pNiceArray}{c|c}[]
  0&  I_{m_2} \\
\hline
I_{m_2} & 0
\end{pNiceArray}
\otimes B_{2m_1}^{\text{Möb}}\\
    +
I_2\otimes I_{m_3} \otimes \mathrm{diag}\left(
    2\iota a_2\sin{\frac{(2\times1-1)\pi}{2m_2}},     \dots ,2\iota a_2\sin{\frac{(2\times2m_2-1)\pi}{2m_2}}\right)\otimes J_{2m_1}\\
+\begin{pmatrix}
    1&0\\
    0 &-1
\end{pmatrix}\otimes
\mathrm{diag}\left(
    2\iota a_3\sin{\frac{(2\times1-1)\pi}{2m_3}},     \dots ,2\iota a_3\sin{\frac{(2m_3-1)\pi}{2m_3}}\right)
    \\
    \otimes
-\mathrm{adiag}\left(
    e^{\iota\frac{\pi}{2m_2}},
    \dots,
    e^{\iota\frac{(2\times2m_2-1)\pi}{2m_2}}
    \right)
    \otimes J_{2m_1}.
\end{multline*}
Observing that
\begin{equation} 
e^{\iota\frac{(2(n-k+1)-1)\pi}{2n}}=-e^{-\iota\frac{(2k-1)\pi}{2n}},
\end{equation}
\begin{equation}
\label{eqn:sin2theta}
\sin{\frac{(2(n-k+1)-1)\pi}{2n}}
=\sin{\frac{(2k-1)\pi}{2n}},
\end{equation}
and applying some simultaneous row and column interchanges on the second matrix of each tensor, turns it into a block diagonal matrix. 
In particular, if $m_3$ is even, 
\begin{multline}
\label{eqn:KG}
\mathcal{K}_G\sim 
    I_2\otimes I_{m_3} \otimes I_{2m_2}\otimes  T_{2m_1}(-a_1,x,a_1)\\
+a_1 \begin{pmatrix}
    1&0\\
    0 &-1
\end{pmatrix}\otimes
\begin{pNiceMatrix}
\Block[borders={bottom,right}]{2-2}{}
0&e^{\iota\frac{\pi}{2m_3}}& && \\
-e^{-\iota\frac{\pi}{2m_3}}&&&\text{\Huge{0}}& \\
&& \Block[]{1-1}{}
\ddots &  & \\
&\text{\Huge{0}}& & \Block[borders={left,top}]{2-2}{}
0&e^{\iota\frac{\left(2\times\frac{m_3}{2}-1\right)\pi}{2m_3}}\\
&&& -e^{-\iota\frac{\left(2\times\frac{m_3}{2}-1\right)\pi}{2m_3}}&0 
\end{pNiceMatrix}
\otimes 
\begin{pNiceArray}{c|c}[]
  0&  I_{m_2} \\
\hline
I_{m_2} & 0
\end{pNiceArray}
\otimes B_{2m_1}^{\text{Möb}}\\
    +
I_2\otimes I_{m_3} \otimes \mathrm{diag}\left(
    2\iota a_2\sin{\frac{(2\times1-1)\pi}{2m_2}},     \dots ,2\iota a_2\sin{\frac{(2\times2m_2-1)\pi}{2m_2}}\right)\otimes J_{2m_1}\\
+\begin{pmatrix}
    1&0\\
    0 &-1
\end{pmatrix}\otimes
2\iota a_3
\mathrm{diag}\left(
 \sin{\frac{\pi}{2m_3}},\sin{\frac{\pi}{2m_3}},\dots,\sin{\frac{(2\times\frac{m_3}{2}-1)\pi}{2m_3}},\sin{\frac{(2\times\frac{m_3}{2}-1)\pi}{2m_3}}
\right)\\
\otimes -\mathrm{adiag}\left(
    e^{\iota\frac{\pi}{2m_2}},
    \dots,
    e^{\iota\frac{(4m_2-1)\pi}{2m_2}}
    \right)
    \otimes J_{2m_1}.
\end{multline}
Since the first matrix of each tensor product in \eqref{eqn:KG} is diagonal and the second matrix is block diagonal, the determinant of the matrix $\mathcal{K}_G$ can be written as the following product.
\begin{multline*}
\det(\mathcal{K}_G)=
\prod_{i_3=1}^{\lfloor\frac{m_3}{2}\rfloor} 
    \det\Bigg(I_2\otimes I_{2} \otimes I_{2m_2}\otimes  T_{2m_1}(-a_1,x,a_1)\\
+a_1 \begin{pmatrix}
    1&0\\
    0 &-1
\end{pmatrix}\otimes
\begin{pmatrix}
    0&e^{\iota\frac{(2i_3-1)\pi}{2m_3}}\\
    -e^{-\iota\frac{(2i_3-1)\pi}{2m_3}}&0
\end{pmatrix}
\otimes 
\begin{pNiceArray}{c|c}[]
  0&  I_{m_2} \\
\hline
I_{m_2} & 0
\end{pNiceArray}
\otimes B_{2m_1}^{\text{Möb}}\\
    +
I_2\otimes I_{2} \otimes \mathrm{diag}\left(
    2\iota a_2\sin{\frac{(2\times1-1)\pi}{2m_2}},     \dots ,2\iota a_2\sin{\frac{(2\times2m_2-1)\pi}{2m_2}}\right)\otimes J_{2m_1}\\
+\begin{pmatrix}
    1&0\\
    0 &-1
\end{pmatrix}\otimes
2\iota a_3\begin{pmatrix}
    \sin{\frac{(2 i_3-1)\pi}{2m_3}}&0\\
    0&\sin{\frac{(2i_3-1)\pi}{2m_3}} 
\end{pmatrix}
\otimes 
-\mathrm{adiag}\left(
    e^{\iota\frac{\pi}{2m_2}},
    \dots,
    e^{\iota\frac{(4m_2-1)\pi}{2m_2}}
    \right)
\otimes J_{2m_1}\Bigg)\\
        \times
\begin{cases}
1 &\text{ if $m_3$ is even,}\\
 \det\Bigg(I_2\otimes I_{2m_2}\otimes  T_{2m_1}(-a_1,x,a_1)
 \\
+ \iota a_1 \begin{pmatrix}
    1&0\\
    0 &-1
\end{pmatrix}\otimes
\begin{pNiceArray}{c|c}[]
  0&  I_{m_2} \\
\hline
I_{m_2} & 0
\end{pNiceArray}
\otimes B_{2m_1}^{\text{Möb}}
&\text{ if $m_3$ is odd.}\\
    +
I_2\otimes \mathrm{diag}\left(
    2\iota a_2\sin{\frac{(2\times1-1)\pi}{2m_2}},     \dots ,2\iota a_2\sin{\frac{(2\times2m_2-1)\pi}{2m_2}}\right)\otimes J_{2m_1}\\
+2\iota a_3\begin{pmatrix}
    1&0\\
    0 &-1
\end{pmatrix}\otimes
-\mathrm{adiag}\left(
    e^{\iota\frac{\pi}{2m_2}},
    \dots,
    e^{\iota\frac{(4m_2-1)\pi}{2m_2}}
    \right)
\otimes J_{2m_1}\Bigg)
\end{cases}
\end{multline*}
Further, $\mathrm{adiag}\left(e^{\iota\frac{(2i_3-1)\pi}{2m_3}},-e^{-\iota\frac{(2i_3-1)\pi}{2m_3}}
\right)$ 
diagonalizes to $\mathrm{diag}\left(\iota,-\iota\right)$. Therefore,
\begin{multline*}
\det(\mathcal{K}_G)=
\prod_{i_3=1}^{\lfloor\frac{m_3}{2}\rfloor} 
    \det(I_2\otimes I_{2} \otimes I_{2m_2}\otimes  T_{2m_1}(-a_1,x,a_1)\\
+a_1 \begin{pmatrix}
    1&0\\
    0 &-1
\end{pmatrix}\otimes
\begin{pmatrix}
    \iota&0\\
    0&-\iota
\end{pmatrix}
\otimes 
\begin{pNiceArray}{c|c}[]
  0&  I_{m_2} \\
\hline
I_{m_2} & 0
\end{pNiceArray}
\otimes B_{2m_1}^{\text{Möb}}\\
    +
I_2\otimes I_{2} \otimes \mathrm{diag}\left(
    2\iota a_2\sin{\frac{(2\times1-1)\pi}{2m_2}},     \dots ,2\iota a_2\sin{\frac{(2\times2m_2-1)\pi}{2m_2}}\right)\otimes J_{2m_1}\\
+\begin{pmatrix}
    1&0\\
    0 &-1
\end{pmatrix}\otimes
2\iota a_3\sin{\frac{(2 i_3-1)\pi}{2m_3}}\begin{pmatrix}
    1&0\\
    0&1
\end{pmatrix}
\otimes 
-\mathrm{adiag}\left(
    e^{\iota\frac{\pi}{2m_2}},
    \dots,
    e^{\iota\frac{(4m_2-1)\pi}{2m_2}}
    \right)
\otimes J_{2m_1})
\end{multline*}
\begin{multline*}
    \times
\begin{cases}
1 &\text{ if $m_3$ is even,}\\
 \det(I_2\otimes I_{2m_2}\otimes  T_{2m_1}(-a_1,x,a_1)
+ \iota a_1 \begin{pmatrix}
    1&0\\
    0 &-1
\end{pmatrix}\otimes
\begin{pNiceArray}{c|c}[]
  0&  I_{m_2} \\
\hline
I_{m_2} & 0
\end{pNiceArray}
\otimes B_{2m_1}^{\text{Möb}}\\
    +
I_2\otimes \mathrm{diag}\left(
    2\iota a_2\sin{\frac{(2\times1-1)\pi}{2m_2}},     \dots ,2\iota a_2\sin{\frac{(2\times2m_2-1)\pi}{2m_2}}\right)\otimes J_{2m_1}\\
+2\iota a_3\begin{pmatrix}
    1&0\\
    0 &-1
\end{pmatrix}\otimes
-\mathrm{adiag}\left(
    e^{\iota\frac{\pi}{2m_2}},
    \dots,
    e^{\iota\frac{(4m_2-1)\pi}{2m_2}}
    \right)
\otimes J_{2m_1})
    &\text{ if $m_3$ is odd.}
\end{cases}
\end{multline*}
Now, since the initial two matrices of each tensor product are $2\times 2$ diagonal matrices, it is enough to determine the determinant of the following matrices:

\begin{multline*}
F_{i_3}^{\alpha,\beta}(\text{ for }\alpha=\pm 1,\beta=\pm 1)
=I_{2m_2}\otimes  T_{2m_1}(-a_1,x,a_1)
+\iota\alpha a_1 
\begin{pNiceArray}{c|c}[]
  0&  I_{m_2} \\
\hline
I_{m_2} & 0
\end{pNiceArray}
\otimes B_{2m_1}^{\text{Möb}}\\
    + \mathrm{diag}\left(
    2\iota a_2\sin{\frac{(2\times1-1)\pi}{2m_2}},     \dots ,2\iota a_2\sin{\frac{(2\times2m_2-1)\pi}{2m_2}}\right)\otimes J_{2m_1}\\
-2\iota \beta a_3\sin{\frac{(2 i_3-1)\pi}{2m_3}}
\mathrm{adiag}\left(
    e^{\iota\frac{\pi}{2m_2}},
    \dots,
    e^{\iota\frac{(4m_2-1)\pi}{2m_2}}
    \right)
\otimes J_{2m_1}.
\end{multline*}
We can rewrite $F_{i_3}^{\alpha,\beta}$ with the help of the trigonometric identities in \eqref{eqn:TriId}
as
\vspace{0cm}
\begin{multline*}
F_{i_3}^{\alpha,\beta}=I_{2}\otimes I_{m_2}\otimes  T_{2m_1}(-a_1,x,a_1)
+\iota \alpha a_1 \begin{pmatrix}
        0&1\\
1&0
    \end{pmatrix}
    \otimes I_{m_2}
\otimes B_{2m_1}^{\text{Möb}}\\
    + \begin{pmatrix}
        1&0\\
0&-1
    \end{pmatrix}
    \otimes \mathrm{diag}\left(
    2\iota a_2\sin{\frac{(2\times1-1)\pi}{2m_2}},     \dots ,2\iota a_2\sin{\frac{(2\times m_2-1)\pi}{2m_2}}\right)\otimes J_{2m_1}\\
-2\iota\beta a_3\sin{\frac{(2 i_3-1)\pi}{2m_3}}
\begin{pmatrix}
        0&1\\
-1&0
    \end{pmatrix}
\otimes
\mathrm{adiag}\left(
    e^{\iota\frac{\pi}{2m_2}},
    \dots,
    e^{\iota\frac{(2m_2-1)\pi}{2m_2}}
    \right)
\otimes J_{2m_1}.
\end{multline*}
Using the unitary transform $\frac{1}{\sqrt{2}}\begin{pmatrix}
    -1&1\\
    1 &1
\end{pmatrix}\otimes I_{m_2}\otimes I_{2m_1}$,
we obtain
\begin{multline*}
F_{i_3}^{\alpha,\beta}\sim I_{2}\otimes I_{m_2}\otimes  T_{2m_1}(-a_1,x,a_1)
+\iota\alpha a_1 \begin{pmatrix}
        -1&0\\
0&1
    \end{pmatrix}
    \otimes I_{m_2}
\otimes B_{2m_1}^{\text{Möb}}\\
    + \begin{pmatrix}
        0&-1\\
-1&0
    \end{pmatrix}
    \otimes \mathrm{diag}\left(
    2\iota a_2\sin{\frac{(2\times1-1)\pi}{2m_2}},     \dots ,2\iota a_2\sin{\frac{(2\times m_2-1)\pi}{2m_2}}\right)\otimes J_{2m_1}\\
-2\iota\beta a_3\sin{\frac{(2 i_3-1)\pi}{2m_3}}
\begin{pmatrix}
        0&-1\\
1&0
    \end{pmatrix}
\otimes
\mathrm{adiag}\left(
    e^{\iota\frac{\pi}{2m_2}},
    \dots,
    e^{\iota\frac{(2m_2-1)\pi}{2m_2}}
    \right)
\otimes J_{2m_1}.
\end{multline*}

Repeating the same procedure on the second 
matrix in each tensor if $m_2$ is even, results in
\begin{multline*}
F_{i_3}^{\alpha,\beta}\sim
I_2\otimes I_{m_2} \otimes  T_{2m_1}(-a_1,x,a_1)
    +\iota \alpha a_1 \begin{pmatrix}
        -1&0\\
0&1
    \end{pmatrix}
    \otimes I_{m_2}
\otimes B_{2m_1}^{\text{Möb}}\\
    + \begin{pmatrix}
        0&-1\\
-1&0
    \end{pmatrix}
    \otimes 
2\iota a_2
\mathrm{diag}\left(\sin{\frac{\pi}{2m_2}},\sin{\frac{\pi}{2m_2}},\dots,
\sin{\frac{(2\times \frac{m_2}{2}-1)\pi}{2m_2}},\sin{\frac{(2\times \frac{m_2}{2}-1)\pi}{2m_2}}
\right)
\otimes J_{2m_1}\\
-2\iota\beta a_3\sin{\frac{(2 i_3-1)\pi}{2m_3}}
\begin{pmatrix}
        0&-1\\
1&0
    \end{pmatrix}
\otimes
\mathrm{diag}\left(\iota,-\iota,\dots,\iota ,-\iota \right)
\otimes J_{2m_1}.
\end{multline*}
Further, notice that $F_{i_3}^{\alpha,\beta}$ is the product of determinants of $2\times 2$ block matrices of the following form
\[
\begin{pmatrix}
    A& B\\
    -\overline{B}& \overline{A}
\end{pmatrix}
\]
where
\begin{align*}
    A&=I_{m_2}\otimes T_{2m_1}(-a_1,x,a_1)
    -
    \iota\alpha a_1 I_{m_2}\otimes B_{2m_1}^{\text{Möb}},
\end{align*}
and
\begin{align*}
B&=\Big(
-\mathrm{diag}\left(2\iota a_2\sin{\frac{\pi}{2m_2}},2\iota a_2\sin{\frac{\pi}{2m_2}},\dots,2\iota a_2\sin{\frac{(2\times \frac{m_2}{2}-1)\pi}{2m_2}},2\iota a_2\sin{\frac{(2\times \frac{m_2}{2}-1)\pi}{2m_2}}\right)\\
&+2\iota\beta a_3\sin{\frac{(2 i_3-1)\pi}{2m_3}}
\mathrm{diag}\left(\iota,-\iota,\dots,\iota ,-\iota \right)
\Big)
\otimes
    J_{2m_{1}}.
\end{align*}
Using a similar argument as in the proof of \Cref{thm:PF3grid_Mob}, we can see that
$ \overline{B}A^*=\overline{A} \,\overline{B}$.
Now, using \Cref{lem:det2by2Block},
we get
\begin{equation}
\label{eqn:detABBconjAconjKlein}
\det\begin{pmatrix}
    A& B\\
    -\overline{B}& \overline{A}
\end{pmatrix}=\det{(AA^*+B{B}^*)}.
\end{equation}
Diagonalizing the matrix $A$ yields the following:
\begin{align*}
 A&=I_{m_2}\otimes \left(T_{2m_1}(-a_1,x,a_1)
    -
   \iota \alpha a_1 B_{2m_1}^{\text{Möb}}\right)   \\
    &\sim I_{m_2}\otimes \mathrm{diag}\left(x-2\iota\alpha a_1\sin{\frac{3\pi}{4m_1}},\dots,x-2\iota\alpha a_1\sin{\frac{(4(2m_1)-1)\pi}{4m_1}}\right)
\end{align*}
and
\begin{equation}
\label{eqn:AAstarKlein}
AA^*\sim I_{m_2}\otimes \mathrm{diag}\bigg(x^2+4a_1^2\sin^2{\frac{3\pi}{4m_1}},\dots,x^2+4a_1^2\sin^2{\frac{(8m_1-1)\pi}{4m_1}}\bigg).
\end{equation}
Further, a quick calculation shows that
\begin{multline}
\label{eqn:BBstarKlein}
BB^{*}=\Bigg(
\mathrm{diag}\left(4a_2^2
\sin^2{\frac{\pi}{2m_2}},4a_2^2\sin^2{\frac{\pi}{2m_2}},\dots,
4a_2^2\sin^2{\frac{(2\times \frac{m_2}{2}-1)\pi}{2m_2}},4a_2^2\sin^2{\frac{(2\times \frac{m_2}{2}-1)\pi}{2m_2}}
\right)\\
+4 a_3^2\sin^2{\frac{(2 i_3-1)\pi}{2m_3}}
I_{m_2}\Bigg)
\otimes I_{2m_1}.
\end{multline}
Combining equations \eqref{eqn:detABBconjAconjKlein}, \eqref{eqn:AAstarKlein}, and \eqref{eqn:BBstarKlein} results
\begin{multline*}
\det{F_{i_3}^{\alpha,\beta}}=\prod_{i_2=1}^{\lfloor \frac{m_2}{2}\rfloor}
\prod_{i_1=1}^{2m_1} \bigg(x^2+4a_1^2\sin^2{\frac{(4i_1-1)\pi}{4m_1}}+4a_2^2\sin^2{\frac{(2i_2-1)\pi}{2m_2}}+4a_3^2\sin^2{\frac{(2i_3-1)\pi}{2m_3}}\bigg)^2\\
\times
\begin{cases}
1&\text{ if $m_2$ is even,}\\
\prod_{i_1=1}^{2m_1} \bigg(x^2+4a_1^2\sin^2{\frac{(4i_1-1)\pi}{4m_1}}+4a_2^2\sin^2{\frac{((m_2+1)-1)\pi}{2m_2}}+4a_3^2\sin^2{\frac{(2i_3-1)\pi}{2m_3}}\bigg)& \text{ if $m_2$ is odd.}
\end{cases}
\end{multline*}
Hence,
\begin{multline*}
\det(\mathcal{K}_G)=
\prod_{i_3=1}^{\lfloor\frac{m_3}{2}\rfloor} 
\prod_{i_2=1}^{m_2}
\prod_{i_1=1}^{2m_1} \bigg(x^2
+4a_1^2\sin^2{\frac{(4i_1-1)\pi}{4m_1}}
+4a_2^2\sin^2{\frac{(2i_2-1)\pi}{2m_2}}
+4a_3^2\sin^2{\frac{(2i_3-1)\pi}{2m_3}}\bigg)^4\\
\times
\begin{cases}
    1&\text{$m_3$ is even,}\\
    \prod_{i_2=1}^{m_2}
\prod_{i_1=1}^{2m_1} \bigg(x^2
+4a_1^2\sin^2{\frac{(4i_1-1)\pi}{4m_1}}
+4a_2^2\sin^2{\frac{(2i_2-1)\pi}{2m_2}}
+4a_3^2\sin^2{\frac{((m_3+1)-1)\pi}{2m_3}}\bigg)^2
&\text{$m_3$ is odd,}
\end{cases}
\end{multline*}
which can be rewritten using the trigonometric identities \eqref{eqn:sin4theta} and \eqref{eqn:sin2theta} as
\begin{multline*}
\det(\mathcal{K}_G)=
\prod_{i_3=1}^{m_3} 
\prod_{i_2=1}^{m_2}
\prod_{i_1=1}^{m_1} \bigg(x^2
+4a_1^2\sin^2{\frac{(4i_1-1)\pi}{4m_1}}
+4a_2^2\sin^2{\frac{(2i_2-1)\pi}{2m_2}}
+4a_3^2\sin^2{\frac{(2i_3-1)\pi}{2m_3}}\bigg)^4.
\end{multline*}
\end{proof}
We conclude the section by providing an example that shows that the formula in \eqref{eqn:2dKlein} can only be generalised to three dimensions and by raising a pertinent question of whether is it possible to define the non-orientable boundary conditions such that the product formulas for higher dimensions still hold while preserving the relationship in \eqref{eqn:Cyl_mob3D_relation}.
    \begin{figure}[h!]
        \centering      \includegraphics[scale=0.85]{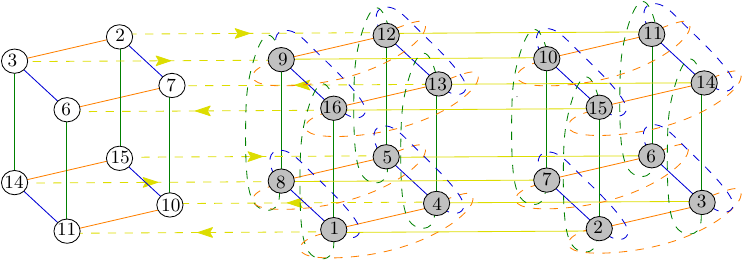}
        \caption{four-dimensional Klein grid graph, $Q^{\text{Klein}}_{2,2,2,2}$.}
    \label{fig:4DKlein}
    \end{figure}
\begin{example}
\label{exm:CEKleinHD}
    Let $G=Q^{\text{Klein}}_{2,2,2,2}$ be the four-dimensional Klein grid graph as shown in \Cref{fig:4DKlein}. Solid edges are oriented from a lower labelled vertex to a higher labelled vertex and dashed edges are oriented as described above.
    Let the vertex weight be $0$ for all the vertices and edge weights be $a_1,a_2,a_3$ and $a_4$ for the edges 
along the different coordinate axes. Then, the partition function of the monopole-dimer model on $G$ is 
\begin{equation*}
    2^{16}(a_1^2 + a_2^2 + a_3^2 + a_4^2)^4(a_2^2 + a_3^2 + a_4^2)^4
\end{equation*}
which is not an $8^{th}$ power. Hence the product formula \eqref{eqn:2dKlein} does not generalise to higher dimensions in obvious way.
\end{example}
\begin{remark}
    Again, note that the oriented $d$-dimensional Klein grid $Q^{\text{Klein}}_{2m_1,\dots,2m_d}$ can be regarded as the oriented Cartesian product of $P_{2m_1},C_{2m_2},\dots,C_{2m_d}$ (oriented from lower-labelled vertex to higher-labelled vertex) together with some additional dashed edges oriented in the specified manner. We believe (without proof) that the partition function of the monopole-dimer model remains unchanged regardless of the Pfaffian orientation on the path and cycle graphs.
\end{remark}
\section{acknowledgements}
The author thanks Prof. Arvind Ayyer for engaging in valuable and insightful discussions. Additionally, thanks are extended to the Prime Minister's Research Fellowship (PM-MHRD\textunderscore19\textunderscore17579) Scheme for providing funding support.

\bibliographystyle{alpha}
\bibliography{bib}

\end{document}